\documentclass[11pt,leqno]{amsart}
\usepackage{amsmath,amssymb,amsthm}
\usepackage{hyperref}

\usepackage{tikz-cd}

\usepackage{hyperref}

\usepackage{bbm}

\newcommand{\R}{\mathbb{R}}

\newcommand{\lspan}{\operatorname{span}}
\newcommand{\clspan}{\overline{\operatorname{span}}}

\renewcommand{\geq}{\geqslant}
\renewcommand{\leq}{\leqslant}

\newcommand{\norm}[1]{\left\Vert#1\right\Vert}

\newcommand{\Lip}{{\mathrm{Lip}}_0}

\newcommand{\dens}{\operatorname{dens}}

\newcommand{\supp}{\operatorname{supp}}

\newcommand{\linf}[1]{\ell_\infty \left( #1  \right)}

\newcommand{\projtensor}{\,\widehat{\bigotimes}_\pi}
\newcommand{\injtensor}{\,\widehat{\bigotimes}_\eps}
\newcommand{\tensor}{\,\widehat{\bigotimes}_\gamma}
\newcommand{\sop}{\operatorname{supp}}

\newcommand{\ASQ}[1]{\text{ASQ}_{<#1}} 
\newcommand{\SQ}[1]{\text{SQ}_{<#1}}

\newtheorem{theorem}{Theorem}[section]
\newtheorem{lemma}[theorem]{Lemma}

\newtheorem{proposition}[theorem]{Proposition}
\newtheorem{corollary}[theorem]{Corollary}
\theoremstyle{definition}
\newtheorem{definition}[theorem]{Definition}
\newtheorem{example}[theorem]{Example}

\theoremstyle{remark}
\newtheorem{remark}[theorem]{Remark}
\newtheorem{question}{Question}
\numberwithin{equation}{section}
\newcommand{\abs}[1]{\left\lvert#1\right\rvert}

\def\fnote#1{\footnote}

\def\ignora#1{}
\def\n3#1{\left\vert  \! \left\vert \! \left\vert \, #1 \, \right\vert \!
  \right\vert \! \right\vert }

\newcommand{\pten}{\ensuremath{\widehat{\otimes}_\pi}}

\renewcommand{\leq}{\le}

\usepackage{accents}
\usepackage[most]{tcolorbox}
\usepackage{comment}

\let\emptyset\varnothing

\newcommand{\N}{\mathbb{N}}

\newcommand{\ra}{\longrightarrow}

\newcommand{\eps}{\varepsilon}

\begin{document}


\author{ Esteban Martínez Vañó }\address{Universidad de Granada, Facultad de Ciencias. Departamento de An\'{a}lisis Matem\'{a}tico, 18071-Granada
(Spain)} \email{ emv@ugr.es}

\author{ Abraham Rueda Zoca }\address{Universidad de Granada, Facultad de Ciencias. Departamento de An\'{a}lisis Matem\'{a}tico, 18071-Granada
(Spain)} \email{ abrahamrueda@ugr.es}
\urladdr{\url{https://arzenglish.wordpress.com}}

\subjclass[2020]{46B26; 54C15; 46B04}

\keywords{transfinite ideals; transfinite almost isometric ideals; spaces of universal disposition; transfinite injective spaces}

\title{Transfinite (almost isometric) ideals in Banach spaces}

\begin{abstract}
We present and study some transfinite versions of (almost isometric) ideals in Banach spaces. As these notions are closely related with Lindenstrauss and Gurari\u{\i} spaces respectively, we will present a similar characterization for transfinite injective spaces and spaces of (almost) universal disposition in terms of these transfinite ideals. Furthermore, we construct several examples outside these type of Banach spaces and make a revision of some classical results for transfinite ideals.
\end{abstract}

\maketitle

\markboth{MART\'INEZ VAÑ\'O,  RUEDA ZOCA}{TRANSFINITE (ALMOST ISOMETRIC) IDEALS IN BANACH SPACES}

\tableofcontents

\section{Introduction}

One of the most longstanding research line in Functional Analysis is related to the extension of certain classes of continuous functions. To mention a couple of celebrated results we may mention, on the one hand, the Hahn-Banach theorem, which has to do with extensions of linear and continuous functionals; on the other hand, McShane extension theorem \cite[Theorem 1.33]{weaver18}, which deals with the problem of extending Lipschitz mappings taking values on $\mathbb R$. The search of a vector-valued version of these results leads to the notion of \textit{injective Banach spaces}. A Banach space $X$ is \textit{injective} if every linear bounded operator $t:Y\longrightarrow X$ admits a norm-preserving extension $T:Z\longrightarrow X$ (where $Y\subseteq Z$). This is equivalent to the fact that $X$ is $1$-complemented in any Banach space containing it (see. e.g. \cite[Section 2.5]{alka2006}). Even though injective Banach spaces enjoy a pleasant property, this class of spaces is quite restrictive (as a sample of this, observe that separable injective Banach spaces are finite dimensional \cite[Corollary 4.3.13]{alka2006}). Because of this, in order to deal with a larger class of Banach spaces, we can consider two different classes of Banach spaces whose extension property is related with finite rank operators. On the one hand, we can consider the Banach spaces $X$ satisfying that given finite dimensional spaces $Y\subseteq Z$ and $t:Y\longrightarrow X$ there exists for every $\varepsilon>0$ an extension $T:Z\longrightarrow X$ with $\Vert T\Vert\leq (1+\varepsilon)\Vert t\Vert$, which give raise to the class of \textit{$L_1$ preduals} (see \cite[Theorem 6.1]{linds64} for explanation). On the other hand, if we consider the Banach spaces $X$ satisfying that given finite dimensional spaces $Y\subseteq Z$ and any isometry $t:Y\longrightarrow X$ there exists for every $\varepsilon>0$ an extension $T:Z\longrightarrow X$ such that $(1-\varepsilon)\Vert z\Vert\leq \Vert T(z)\Vert\leq (1+\varepsilon)\Vert z\Vert$, we get the class of \textit{Gurari\u{\i} spaces} (see the survey \cite{gk11} for background and more references). 

As well as happen with injective Banach spaces, the property of being an $L_1$ predual or a Gurari\u{\i} space can be characterised in terms of a certain projective property which depends on finite dimensional subspaces. In order to describe them, we need the following definition.

\begin{definition}\label{defideal}
Let $X$ be a Banach space and $Y\subseteq X$ a subspace. We say that $Y$ is:
\begin{enumerate}
    \item an \textit{ideal} in $X$ if, for any $E\subseteq X$ of finite dimension and any $\varepsilon>0$, there exists $T:E\longrightarrow Y$ such that $\Vert T\Vert \le 1+\varepsilon$ and $T(e)=e$ holds for every $e\in E\cap Y$.
    \item an \textit{almost isometric ideal} in $X$ if, for any $E\subseteq X$ of finite dimension and any $\varepsilon>0$, there exists $T:E\longrightarrow Y$ such that $(1-\varepsilon)\Vert e\Vert\le \Vert T(e)\Vert\le (1+\varepsilon)\Vert e\Vert$ holds for every $e\in E$ and $T(e)=e$ holds for every $e\in E\cap Y$.
    \end{enumerate}
\end{definition}
We refer to \cite{gks93} and references therein for background on ideals, whereas we send the interested reader on almost isometric ideals to \cite{aln2}. 

Going back to the search of a projective characterisation of $L_1$ preduals and of Gurari\u{\i} spaces, we get the following two results:
\begin{enumerate}
    \item A Banach space $X$ is an $L_1$ predual if, and only if, $X$ is an ideal in any Banach space containing it \cite[Proposition 3.4]{fakh72}.
    \item A Banach space $X$ is a Gurari\u{\i} space if, and only if, $X$ is an almost isometric ideal in any Banach space containing it \cite[Theorem 4.3]{aln2}.
\end{enumerate}

Generalisations of $L_1$ preduals and Gurari\u{\i} spaces can be considered if we strengthen the extension property by considering domains of higher density character: 

Given a Banach space $X$ and an infinite cardinal $\kappa$, we say that $X$ is
\begin{enumerate}
    \item \textit{$\kappa$ injective} if for any two normed spaces $Y\subseteq Z$ with $\dens(Z)< \kappa$ and any linear bounded operator $t:Y\longrightarrow X$ there exists a norm-preserving extension $T:Z\longrightarrow X$.
    \item \textit{of universal disposition (resp. almost universal disposition) for density $\kappa$} ($\text{(A)UD}_{<\kappa}$ in short) if for any two normed spaces $Y\subseteq Z$ with $\dens(Z)< \kappa$ and any isometry $t:Y\longrightarrow X$ there exists (for every $\varepsilon>0$) an extension $T:Z\longrightarrow X$ which is an isometry (respectively such that $(1-\varepsilon)\Vert z\Vert\le \Vert T(z)\Vert\le (1+\varepsilon)\Vert z\Vert$ holds for every $z\in Z$).
    \end{enumerate}
We refer the reader to \cite{accgm15} and references therein for background on $\kappa$ injective Banach spaces and to \cite{gk11} and references therein for background about spaces of (almost) universal disposition.

At this point we consider the following questions: are there transfinite notions of ideals and almost isometric ideals which produce a projective characterisation of $\kappa$ injective spaces and of spaces of (almost) universal disposition for density $\kappa$? The aim of this paper is to introduce such transfinite notions.

The content is organised as follows. In Section~\ref{section:preliminaries} we introduce all the notation that we will need throughout the text. In Section~\ref{sect:transideal} we develop our research about $\kappa$ ideals in Banach spaces, which are formally defined in Definition~\ref{defi:transideal}. We prove that the notion of $\kappa$ ideal is the transfinite notion of ideals we were looking for characterising $\kappa$ injective Banach spaces in the following sense: a Banach space $X$ is $\kappa$ injective if, and only if, $X$ is a $\kappa$ ideal in every Banach space containing $X$ isometrically (Theorem~\ref{charinj1}). In the search of new examples we prove in Theorem~\ref{tensor} that if $E$ is a $\kappa$ ideal in $X$ and $F$ is a $\kappa$ ideal in $Y$ then $E\projtensor F$ (resp. $E\injtensor F$) is a $\kappa$ ideal in $X\projtensor Y$ (resp. $X\injtensor Y$). We also prove that if $Y$ is a $\kappa$ ideal in $X$ then $\mathcal F(Y)$ is a $\kappa$ ideal in $\mathcal F(X)$. These two results show examples of $\kappa$ ideals which do not come from the class of $\kappa$ injective Banach spaces (indeed they are not even $L_1$ preduals, see Example~\ref{exam:contratensor} and the paragraph after Proposition~\ref{theo:idealfree}).

In Section~\ref{sect:transaiideal} we move to the introduction of $\kappa$ (almost) isometric ideals in Definition~\ref{defi:transaiideal}. We prove that this notion characterises the spaces of (almost) universal disposition in the following sense: a Banach space $X$ is a $\text{(A)UD}_{<\kappa}$ space if, and only if, $X$ is a $\kappa$ (almost) isometric ideal in every Banach space that contains $X$ isometrically (Theorem~\ref{AUDchar}). In order to provide more examples, we prove in Theorem~\ref{theo:injetensortransai} that $E$ is a $\kappa$ (almost) isometric ideal in $X$ and $F$ is a $\kappa$ almost isometric ideal in $Y$ then $E\injtensor F$ is a $\kappa$ almost isometric ideal in $X\injtensor Y$. We devote the rest of the section to establish deep connection between the notion of $\kappa$ (almost) isometric ideals and transfinite version of almost squareness and octahedral norms. For instance, we prove that (A)SQ$_{<\kappa}$ is a property inherited by $\kappa$ (almost) isometric ideals (see Propositions~\ref{teo:hereaiidealesasqlargos} and \ref{prop:heredaiisosq}), obtaining a similar version for the property of $<\kappa$ octahedrality and the failure of $(-1)$-BCP$_{<\kappa}$ in Proposition~\ref{bajaocta}. We also prove in Theorems~\ref{theo:caraasqlargosaiideales} and \ref{theo:caraasqlargosaiidealesbis} that, for a Banach space $X$, $X$ is (A)SQ$_{<\kappa}$ (resp. it fails $(-1)$-BCP$_{<\kappa}$ or is $<\kappa$ octahedral) if, and only if, $X\times\{0\}$ is a $\kappa$ (almost) isometric ideal in $X\oplus_\infty\mathbb R$ (resp. in $X\oplus_1 \mathbb R$). As a consequence, we prove in Corollary~\ref{cor:propUDoctaasq} that if a Banach space $X$ is an $\text{(A)UD}_{<\kappa}$ space then $X$ is simultaneously SQ$_{<\kappa}$ and fails the $(-1)$-BCP$_{<\kappa}$ (resp. $X$ is simultaneously ASQ$_{<\kappa}$ and $<\kappa$ octahedral).

In Section~\ref{sect:classicalresults} we analyse the question whether classical results concerning almost isometric ideals still holds true for the transfinite versions we include in the paper, namely, the principle of local reflexivity, Sims-Yost theorem and Abrahamsen theorem (see Theorems~\ref{theo:plr}, \ref{theo:simyost} and \ref{theo:abrahamsenai}). In Subsection~\ref{subsect:principlelocal} we provide two examples which show that no transfinite version of the principle of local reflexivity is true. Concerning Sims-Yost and Abrahamsen theorems, we prove in Subsection~\ref{subsect:simyostfalsos} that the existence of counterexamples for both results is consistent with ZFC.

Section~\ref{section:aiidealestensorultrapower} is devoted with the study of almost isometric ideals in tensor product spaces and in ultrapower spaces. We have separated these results from the main section of $\kappa$ (almost) isometric ideal since the results of Section~\ref{section:aiidealestensorultrapower} make use of perturbation arguments which are exclusive of the notion of almost isometric ideal and whose inclusion in Section~\ref{sect:transaiideal} would break the transfinite spirit that this section has. The main result in this section is Theorem~\ref{theorem:aiidealinyect}, where it is proved that if  $Z\subseteq X$ and $W\subseteq Y$ are almost isometric ideals, then $Z\injtensor W$ is an almost isometric ideal in $X\injtensor Y$. Concerning the projective tensor product, we show in Example~\ref{exam:contraproyect} that if $Y$ is an almost isometric ideal in $X$ and $Z$ is any Banach space then $Y\projtensor Z$ is not necessarily an almost isometric ideal in $X\projtensor Y$. Up to the authors knowledge, both results were not previously known. We conclude Section~\ref{section:aiidealestensorultrapower} with a result concerning ultrapowers of Banach spaces. More precisely, we prove in Proposition~\ref{prop:ultrapoweraiideals} that if $Y$ is a subspace of $X$ then $Y$ is an almost isometric ideal in $X$ if, and only if, $Y_\mathcal U$ is an isometric ideal in $X_\mathcal U$ for any countably incomplete ultrafilter $\mathcal U$ over any infinite set $I$.

We conclude the paper with Section~\ref{sect:openquestion}, which is devoted to collect final remarks, open questions derived from our research and possible future directions from our work.

\section{Notation and preliminary results}\label{section:preliminaries}

Unless we state the contrary we will consider real Banach spaces. Given a Banach space $X$, we denote by $B_X$ and $S_X$ the closed unit ball and the unit sphere respectively. We also denote by $X^*$ the topological dual of $X$. For a given subset $E$ of $X$ we write $\lspan(E)$ (resp. $\clspan(E)$) for the (closed) linear span of $E$ and $\vert E\vert$ to denote the cardinal of the set $E$. Given two Banach spaces $X$ and $Y$, we denote by $L(X,Y)$ the space of all bounded operators from $X$ to $Y$. Given an $\eps > 0$ we will say that $T \in L(X,Y)$ is an $\eps$-isometry if for any $x \in X$ we have the following inequalities
$$(1-\eps) \norm{x} \le \norm{T(x)} \le (1+\eps) \norm{x}.$$
Finally, we will say that a Banach space $X$ has density $\kappa$, denoted by $\dens(X)$, if it is the least cardinal that a subset spanning a dense subspace of $X$ can have. With this definition we recover the classical one for infinite dimensional Banach spaces, that is, the least cardinal of a dense subset of $X$, but have the advantage that for a finite dimensional Banach space $X$ we obtain that $\dens(X) = \dim(X)$.

As we use some set theoretic notions about cardinals along the text, we refer the reader to \cite[Chapters 3 and 5]{jech03} where one can find more than enough information.

We also refer the interested reader to \cite{abrahamsen15,aln2,rao16,rao13,rao01} and references therein for background on (almost isometric) ideals in Banach spaces.  Let us introduce here, for easy reference, the establishment of the two main results concerning (almost) isometric ideals. To begin with, we present the principle of local reflexivity, whose proof can be found for instance in \cite[Theorem 9.15]{fab}.

\begin{theorem}[Principle of local reflexivity]\label{theo:plr}
Let $X$ be a Banach space. Then, for every finite dimensional subspace $E$ of $X^{**}$, every finite dimensional $F\subseteq X^*$ and every $\varepsilon>0$ there exists a linear operator $T:E\longrightarrow X$ such that
\begin{enumerate}
    \item $T(e)=e$ holds for every $e\in E\cap X$, 
    \item $(1-\varepsilon)\Vert e^{**}\Vert\le \Vert T(e^{**})\Vert\le (1+\varepsilon)\Vert e^{**}\Vert$ holds for every $e^{**}\in E$ and,
    \item $x^*(T(e^{**}))=e^{**}(x^*)$ holds for every $e^{**}\in E$ and every $x^*\in F$. 
\end{enumerate}
\end{theorem}

The second of our results is a classical theorem of Sims and Yost, which exhibits an abundance of ideals in any Banach space $X$ in the following sense.

\begin{theorem}[Sims-Yost, \cite{simyos}]\label{theo:simyost}
Let $X$ be a Banach space and let $Y\subseteq X$ be a subspace. Then there exists $Z\subseteq X$ such that $Z$ is an ideal in $X$, $Y\subseteq Z\subseteq X$ and $\dens(Z)=\dens(Y)$.    
\end{theorem}

A generalisation of the above theorem was proved by T. A. Abrahamsen in \cite[Theorem 1.5]{abrahamsen15}, whose establishment is included for easy reference.

\begin{theorem}[Abrahamsen, \cite{abrahamsen15}]\label{theo:abrahamsenai}
Let $X$ be a Banach space and let $Y\subseteq X$ be a subspace. Then there exists $Z\subseteq X$ such that $Z$ is an almost isometric ideal in $X$, $Y\subseteq Z\subseteq X$ and $\dens(Z)=\dens(Y)$.    
\end{theorem}

\subsection{\texorpdfstring{$\kappa$}{kappa} injective spaces and spaces of universal disposition}

We refer the reader to \cite{accgm15,sepibook} for background on $\kappa$-injective Banach spaces. We present next the most classical known example:

\begin{example}\label{ex:linfsupinjeaco} Given any uncountable regular cardinal $\kappa$, the spaces of the form
$$\ell_\infty(\kappa,I):=\{x\in \ell_\infty(I): \vert \{i\in I: x(i)\neq 0\}\vert< \kappa\},$$
where $\kappa<\vert I\vert$, are $\kappa$ injective, but fail to be $\beta$ injective for any $\kappa<\beta$.

Let us start by proving that $\ell_\infty(\kappa,\Gamma)$ is $\kappa$ injective, for which we consider a bounded operator $t:Y\longrightarrow \ell_\infty(\kappa,I)$ where $\dens(Y)< \kappa$ and let $Z$ be a space containing $Y$. Let us prove that there exists a bounded operator $T:Z\longrightarrow \ell_\infty(\kappa,I)$ extending $t$ such that $\Vert T\Vert=\Vert t\Vert$.

Indeed, take a dense set $\{y_\alpha: \alpha<\dens(Y)\}\subseteq Y$ and consider
$$A:=\bigcup\limits_{\alpha<\dens(Y)} \{i\in I: T(y_\alpha)(i)\neq 0\}\subseteq I,$$
which is a set such that $\vert A\vert< \kappa$ since $\vert \{i\in I: T(y_\alpha)(i)\neq 0\}\vert< \kappa$. Given $i\in A$ consider $f_i:Y\longrightarrow \mathbb R$ defined by $f_i(y):=t(y)(i)$. The functionals  $f_i$ are linear and continuous (and moreover it is not difficult to see that $\Vert T\Vert=\sup\limits_{i\in A}\Vert f_i\Vert$), so by the Hahn-Banach theorem there exists some $F_i:Z\longrightarrow \mathbb R$ extending $f_i$ and such that $\Vert F_i\Vert=\Vert f_i\Vert$ holds for every $i\in A$. Define $T:Z\longrightarrow \ell_\infty(\kappa,I)$ by the equation
$$T(z)(i):=\left\{\begin{array}{cc}
    F_i(z) & i\in A \\
    0 & i\notin A. 
\end{array}\right.$$
It is immediate that $T$ is well defined since $\vert A\vert< \kappa$. It is also immediate that $T(y)=t(y)$ for any $y \in Y$ since $F_i(y)=f_i(y)$ holds for every $i\in A$ and that $\Vert \hat T\Vert=\sup_{i\in A}\Vert F_i\Vert=\sup_{i\in I} \Vert f_i\Vert=\Vert T\Vert$.

In order to prove that $\ell_\infty(\kappa,I)$ is not $\beta$ injective for any cardinal $\beta$ with $\kappa<\beta$, select one such $\beta$ and consider, for any $\alpha<\kappa$, the closed ball $B_\alpha:=B(2e_\alpha,1)$ where we see $\kappa\subseteq I$ by the assumption $\kappa <\vert I\vert$. The family of balls $\{B_\alpha: \alpha<\kappa\}\cup\{B(0,1)\}$ are pairwise intersecting (observe that $e_\alpha+e_{\alpha'}\in B_\alpha\cap B_{\alpha'}$), but we claim that
$$\bigcap\limits_{\alpha<\kappa}B_\alpha\cap B(0,1)=\emptyset.$$
Indeed, if $x\in \bigcap_{\alpha<\kappa}B_\alpha\cap B(0,1)$ we would get $1\geq\Vert 2e_\alpha-x\Vert\geq \vert 2-x(\alpha)\vert=2-x(\alpha)$, from where $x(\alpha)\geq 1$ for every $\alpha<\kappa$ and hence $\kappa\le \vert \{i\in I: x(i)\neq 0\}\vert$, which is impossible. Consequently, we have constructed a family of balls of cardinality $\kappa$ which are pairwise intersecting but whose intersection is empty. By \cite[Proposition 5.12]{sepibook} the space $\ell_\infty(\kappa,I)$ is not $\beta$ injective, as desired.
\end{example}

\begin{remark}\label{rema:examkappasupfunctions} Observe that, in the above example, we have gone beyond and proved that $\ell_\infty(\kappa,I)$ is indeed a \textit{universally $\kappa$ injective} Banach space according to \cite[Definition 1.1]{accgm15}. We refer the interested reader to \cite{accgm15,sepibook} for background.
\end{remark}

Concerning spaces of almost universal disposition we will only need a consequence of \cite[Theorem 4.2]{gk11}: given any Banach space $X$ and any uncountable cardinal $\kappa$ there exists a Banach space $Y$ of universal disposition for density $\kappa$ such that $X\subseteq Y$.

\subsection{Almost square and octahedral Banach spaces}

We will follow the notation of \cite{lan}. Given a Banach space $X$, we say that $X$ is:
\begin{enumerate}
    \item \textit{octahedral} if, given any finite dimensional subspace $Y$ of $X$ and any $\varepsilon>0$, there exists some $x\in S_X$ such that
    $$\Vert y+\lambda x\Vert\geq (1-\varepsilon)(\Vert y\Vert+\vert\lambda\vert)\ \forall y\in Y, \forall \lambda\in\mathbb R.$$
    \item \textit{almost square} (ASQ) if, given any finite dimensional subspace $Y$ of $X$ and any $\varepsilon>0$, there exists some $x\in S_X$ such that
    $$\Vert y+\lambda x\Vert\le (1+\varepsilon)\max\{\Vert y\Vert,\vert\lambda\vert\}\ \forall y\in Y, \forall \lambda\in\mathbb R.$$
\end{enumerate}

The notion of octahedral norms goes back to the end of the eighties and turns out to characterise the containment of $\ell_1$ in the following terms: a Banach space $X$ contains an isomorphic copy of $\ell_1$ if, and only if, there exists an equivalent octahedral norm on $X$ \cite{god89}. The notion of ASQ spaces is quite more recent and goes back to \cite{all16}, where the authors aimed to analyse this property as a strengthening of the fact that every slice of the unit ball has diameter 2. Moreover, it turns out that, in the same way that octahedrality characterises the presence of $\ell_1$, ASQ spaces describes the containment of $c_0$ in the following sense: a Banach space $X$ contains an isomorphic copy of $c_0$ if, and only if, $X$ admits an equivalent ASQ renorming \cite[Corollary 2.4]{blr16}.

We refer the reader to \cite{all16,acllr23,blr16,cll23,lan} for background on octahedral and ASQ spaces.

\subsection{Tensor product notation}\label{subsc:tensorproduct}

The projective tensor product of $X$ and $Y$, denoted by $X \projtensor Y$, is the completion of the algebraic tensor product $X \otimes Y$ endowed with the norm
$$
\|z\|_{\pi} := \inf \left\{ \sum_{n=1}^k \|x_n\| \|y_n\|: z = \sum_{n=1}^k x_n \otimes y_n \right\},$$
where the infimum is taken over all such representations of $z$. The reason for taking completion is that $X\otimes Y$ endowed with the projective norm is complete if, and only if, either $X$ or $Y$ is finite dimensional (see \cite[P.43, Exercises 2.4 and 2.5]{ryan}).

It is well known that $\|x \otimes y\|_{\pi} = \|x\| \|y\|$ for every $x \in X$, $y \in Y$, and that the closed unit ball of $X \projtensor Y$ is the closed convex hull of the set $B_X \otimes B_Y = \{ x \otimes y: x \in B_X, y \in B_Y \}$. Throughout the paper, we will use of both facts without any explicit reference.

Observe that the action of a bounded operator $G\colon X \longrightarrow Y^*$ as a linear functional on $X \projtensor Y$ is given by
$$
G \left( \sum_{n=1}^{k} x_n \otimes y_n \right) = \sum_{n=1}^{k} G(x_n)(y_n),$$
for every $\sum_{n=1}^{k} x_n \otimes y_n \in X \otimes Y$. This action establishes a linear isometry from $L(X,Y^*)$ onto $(X\projtensor Y)^*$ (see e.g. \cite[Theorem 2.9]{ryan}). All along this paper we will use the isometric identification $(X\projtensor Y)^*=  L(X,Y^*)$ without any explicit mention.

Observe also that given two bounded operators $T\colon X\longrightarrow Z$ and $S\colon Y\longrightarrow W$, we can define an operator $T\otimes S\colon X\projtensor Y\longrightarrow Z\projtensor W$ by the action $(T\otimes S)(x\otimes y):=T(x)\otimes S(y)$ for $x\in X$ and $y\in Y$. It follows that $\Vert T\otimes S\Vert=\Vert T\Vert\Vert S\Vert$. Moreover, it is known that if $T,S$ are bounded projections then so is $T\otimes S$. Consequently, if $Z\subseteq X$ is a $1$-complemented subspace, then $Z\projtensor Y$ is a $1$-complemented subspace of $X\projtensor Y$ in the natural way (see \cite[Proposition 2.4]{ryan} for details). A generalisation of this fact is that  if $Z\subseteq X$ is an ideal, then $Z\projtensor Y$ is an ideal in $X\projtensor Y$ in the natural way (see \cite[Theorem 1]{rao01}).


Recall that given two Banach spaces $X$ and $Y$, the
\textit{injective tensor product} of $X$ and $Y$, denoted by
$X \injtensor Y$, is the completion of $X\otimes Y$ under the norm given by
\begin{equation*}
   \Vert u\Vert_{\varepsilon}:=\sup
   \left\{
      \sum_{i=1}^n \vert x^*(x_i)y^*(y_i)\vert
      : x^*\in S_{X^*}, y^*\in S_{Y^*}
   \right\},
\end{equation*}
where $u=\sum_{i=1}^n x_i\otimes y_i$ (see \cite[Chapter 3]{ryan} for background).
Every $u \in X \injtensor Y$ can be viewed as an operator $T_u : X^* \longrightarrow Y$ which is weak$^*$-weak continuous. Under this point of view, the norm on the injective tensor product is nothing but the operator norm.

As well as happen with the projective tensor product, given two bounded operators $T\colon X\longrightarrow Z$ and $S\colon Y\longrightarrow W$, we can define an operator $T\otimes S\colon X\injtensor Y\longrightarrow Z\injtensor W$ by the action $(T\otimes S)(x\otimes y):=T(x)\otimes S(y)$ for $x\in X$ and $y\in Y$. It follows that $\Vert T\otimes S\Vert=\Vert T\Vert\Vert S\Vert$. However, this time we get that if $T,S$ are linear isometries then $T\otimes S$ is also a linear isometry (c.f. e.g. \cite[Section 3.2]{ryan}). This fact is commonly known as ``the injective tensor product respects subspaces''.


\subsection{Lipschitz-free spaces}\label{subsc:lipschitzfree}

A \emph{pointed metric space} is just a metric space $M$ in which we distinguish an element, denoted by $0$. Given a pointed metric space $M$, we write $\Lip(M)$ to denote the Banach space of all Lipschitz maps $f:M\longrightarrow \mathbb R$ which vanish at $0$, endowed with the Lipschitz norm defined by
$$ \| f \| := \sup\left\{\frac{f(x)-f(y)}{d(x,y)} \colon x,y\in M,\, x \neq y \right\}.$$
We denote by $\delta$ the canonical isometric embedding of $M$ into $\Lip(M)^*$ which is given by $\langle f, \delta(x) \rangle =f(x)$ for $x \in M$ and $f \in \Lip(M)$. The closed linear span of $\delta(M)$ in the dual space $\Lip(M)^*$ is denoted by $\mathcal F(M)$ and called \textit{the Lipschitz-free space over $M$}. See the papers \cite{god15} and \cite{gk}, the PhD dissertation \cite{aliaga21} and the book \cite{weaver18} (where it receives the name of Arens-Eells space) for background on these spaces. It is well known that $\mathcal{F}(M)$ is an isometric predual of the space $\Lip(M)$ \cite[pp. 91]{god15}. We call \textit{molecule in $\mathcal F(M)$} any element of the form
$$m_{x,y}:=\frac{\delta(x)-\delta(y)}{d(x,y)}$$
for $x,y\in M$ with $x\neq y$.

Let us recall that when $M$ and $N$ are pointed metric spaces, it is well known that every Lipschitz function $f \colon N \longrightarrow M$ which preserves the origin can be isometrically identified with the continuous linear operator $\widehat{f} \colon \mathcal{F}(N) \longrightarrow \mathcal F(M)$ characterised by $\widehat{f}(\delta(p))=\delta(f(p))$ for every $p \in N$. 

Observe that it is well known (c.f. e.g. \cite[Lemma 2.1]{aliper20}) that the elements of $\mathcal F(M)$ consists of all functionals $\gamma\in\Lip(X)^*$ of the form
$$
\gamma=\sum_{n=1}^\infty\lambda_n m_{x_n,y_n},
$$
with $\left(\lambda_n\right)_{n=1}^\infty\in\ell_1$ and $x_n,y_n\in M$ are such that $x_n\neq y_n$ holds for every $n\in\mathbb N$. Moreover,
$$
\left\|\gamma\right\|=\inf\left\{\sum_{n=1}^\infty\left|\lambda_n\right|\colon \gamma=\sum_{n=1}^\infty\lambda_n m_{x_n,y_n},\, \left(\lambda_n\right)_{n=1}^\infty\in\ell_1,\, x_n\neq y_n\in M\right\}.
$$

\subsection{Ultrapowers of Banach spaces}\label{subsect:ultrapowers}

Given a Banach space $X$ and an infinite set $I$, we denote  $\ell_\infty(I,X):=\{f\colon I\longrightarrow X: \sup_{i\in I}\Vert f(i)\Vert<\infty\}$. Given a free ultrafilter $\mathcal U$ over $I$, consider $N_\mathcal U:=\{f\in \ell_\infty(I,X): \lim_\mathcal U \Vert f(i)\Vert=0\}$. The \textit{ultrapower of $X$ with respect to $\mathcal U$} is
the Banach space
$$X_\mathcal U:=\ell_\infty(I,X)/N_\mathcal U.$$
We will naturally identify a bounded function $f\colon I\longrightarrow X$ with the element $(f(i))_{i\in I}$. In this way, we denote by $[x_i]_{\mathcal U,i}$ or simply by $[x_i]_\mathcal U$, if no confusion is possible, the coset in $X_\mathcal U$ given by $(x_i)_{i\in I}+N_\mathcal U$.

From the definition of the quotient norm, it is not difficult to prove that $\Vert [x_i]_\mathcal U\Vert=\lim_\mathcal U \Vert x_i\Vert$ for every $[x_i]_\mathcal U\in X_\mathcal U$. This implies that the canonical inclusion $j:X\longrightarrow X_\mathcal U$ given by the equation
$$j(x):=[x]_\mathcal U$$
is an into linear isometry. It is well known (c.f. e.g. Propositions 6.1 and 6.2 in \cite{hein80}) that $X_\mathcal U$ is finitely representable in $j(X)$. However, from an inspection of the proofs and the operators used there it can be concluded that $j(X)$ is indeed an almost isometric ideal in $X_\mathcal U$.

We refer the reader to \cite[Section 4.1]{sepibook} for background on ultrapowers of Banach spaces.

\section{Transfinite ideals}\label{sect:transideal}

We begin with the formal establishment of the notion of $\kappa$ ideals in Banach spaces.

\begin{definition}\label{defi:transideal}
    Let $X$ be a Banach space, $Y \subset X$ be a subspace of $X$ and $\kappa$ an infinite cardinal. We say that $Y$ is a \textit{$\kappa$ ideal in $X$} if for every subspace $E \subset X$ with $\dens(E) < \kappa$ there exists a linear operator $P: E \ra Y$ with $\norm{P} = 1$ such that
    $$P(e) = e, \, \forall e \in E \cap Y.$$
\end{definition}

\begin{remark}
Observe that $\aleph_0$ ideals are not the classical ideals, but the ones when one consider $\eps = 0$ in Definition \ref{defideal} (1). In a similar fashion we could have defined an ``almost version'' for $\kappa$ ideals in the following sense:

Given $Y\subseteq X$ we say that $Y$ is a \textit{$\kappa$ almost ideal in $X$} if for every subspace $E \subset X$ with $\dens(E) < \kappa$ and every $\varepsilon>0$ there exists a linear operator $P: E \ra Y$ with $\norm{P} \le 1+\varepsilon$ such that $P(e) = e$ for every $e \in E \cap Y$.

In that case, $\aleph_0$ almost ideals would be the classical ideals, but the reason for not considering this (formally) weaker notion is two-fold: on the one hand, we aimed to characterize $\kappa$ injective spaces and $\kappa$ ideals were more suitable for that and, on the other hand, we did not find any example of spaces $Y\subseteq X$ such that $Y$ is $\kappa$ almost ideal in $X$ but which is not a $\kappa$ ideal in $X$ out of the case that $\kappa=\aleph_0$.
\end{remark}

Next we will show that the notion of $\kappa$ ideal in Banach spaces is the projective characterisation of $\kappa$ injectivity we were looking for. 

\begin{theorem}\label{charinj1}
Let $X$ be a Banach space and $\kappa$ an uncountable cardinal. The following are equivalent:
\begin{enumerate}
    \item $X$ is $\kappa$ injective.
    \item $X$ is a $\kappa$ ideal in every Banach space that contains it.
\end{enumerate}    
\end{theorem}

Compare the above result with \cite[Proposition 3.4]{fakh72}, where it is proved that a Banach space $X$ is an $L_1$ predual if, and only if, $X$ is an ideal in any Banach space containing it.

The proof is analogous to the classical one for injective spaces, but we will include it for completeness.

\begin{proof}
We first suppose (1), that is, that $X$ is $\kappa$ injective, and we will show that if $Y$ is a Banach space such that $Y \supset X$, then $X$ is a $\kappa$ ideal in $Y$.

If $E \subset Y$ with $\dens(E) < \kappa$ and $E \cap X = \{0\}$, then we can take $P: E \ra X$ as any norm $1$ linear operator. So let us suppose that $E \cap X$ is not the trivial space.

In this case, denoting by $\iota: E \cap X \ra X$ the inclusion operator, as $\dens(E \cap X) \le \dens(E) < \kappa$, we can extend it to an operator $P: E \ra X$ with $\norm{P} = \norm{\iota} = 1$ and it will obviously be the desired projection.

For the other implication, assume that (2) holds and let us prove that given two normed spaces $Y \subset Z$ such that $\dens(Z) < \kappa$ and an operator $t: Y \ra X$, we can extend it to some operator $T: Z \ra X$ with the same norm.

If $t = 0$ it is clear that we can extend it to the null operator from $Z$ to $X$ and it will preserve the norm, so let us consider that $t$ is not null. Then, we can find a Banach space $W$ which is linearly isometric to some $\linf{\Gamma}$ and such that $X \subset W$, and define $t': Y \ra W$ as $t' = \iota \circ t$ with $\iota: X \ra W$ the inclusion operator. As $W$ is injective, we can extend $t'$ to some operator $T': Z \ra W$ in such a way that $\norm{T'} = \norm{t'}$.

By denoting $Z' = T'(Z) \subset W$ we have that $\dens(Z') \le \dens(Z) < \kappa$ and by (2) we obtain that there is some linear operator $P: Z' \ra X$ with $\norm{P} = 1$ such that $P(z) = z$ for any $z \in Z' \cap X$. By defining $T: Z \ra X$ as $T= P \circ T'$ we obtain the desired extension.
\end{proof}

The connection between the property of being an $L_1$ predual and the notion of ideals in Banach spaces goes further on thanks to \cite[Theorem 1]{rao13} (the proof is based on a characterisation of $L_1$ preduals in terms of a property of intersection of balls, see also \cite[Lemma 2.1]{rueda21}, where the result is proved making use of a property of extensions of operators), where it is proved that if $X$ is an $L_1$ predual and $Y$ is an ideal in $X$ then the space $Y$ is an $L_1$ predual too. Making an adaptation of the argument of \cite[Lemma 2.1]{rueda21}, we prove that a $\kappa$ ideal in a $\kappa$ injective Banach space is a $\kappa$ injective Banach space.

\begin{theorem}\label{injsubsp}
Let $\kappa$ an infinite cardinal, $X$ be a $\kappa$ injective space and $Y$ be a subspace of $X$. Then, $Y$ is $\kappa$ injective if, and only if, it is a $\kappa$ ideal in $X$.    
\end{theorem}

\begin{proof}
If $Y$ is $\kappa$ injective (so it is complete), then it is a $\kappa$ ideal in $X$ by Theorem~\ref{charinj1}, so it is enough to prove the other implication, that is, we must see that given two  normed spaces $Z \subset W$ with $\dens(W) < \kappa$ and a linear operator $t: Z \ra Y$, there exits an extension $T: W \ra Y$ with $\norm{T} = \norm{t}$.

As $X$ is $\kappa$ injective we can extend the operator $t' = \iota \circ t : Z \ra X$, where $\iota: Y \ra X$ stands for the inclusion operator, to an operator $T' = W \ra X$ with $\norm{T'} = \norm{t'}$. Furthermore, as $Y$ is a $\kappa$ ideal in $X$ and
$$\dens(T'(W)) \le \dens(W) < \kappa,$$
there exists an operator $P: T'(W) \ra Y$ of norm $1$ that fixes the elements in $T'(W) \cap Y$. By defining $T = P \circ T'$ we obtain the desired extension.
\end{proof}

Bearing the above results in mind we get the following corollary, which sums up a characterisation of $\kappa$ injective Banach spaces.

\begin{corollary}
Let $X$ be a Banach space and $\kappa$ be an infinite cardinal. The following assertions are equivalent:
\begin{enumerate}
    \item $X$ is $\kappa$ injective.
    \item $X$ is a $\kappa$ ideal in every Banach space $X$ containing it.
    \item $X$ is a $\kappa$ ideal in every $\kappa$ injective Banach space containing it.
    \item $X$ is an $L_1$ predual and it is a $\kappa$ ideal in $X^{**}$.
\end{enumerate}
\end{corollary}

\begin{proof}
(1)$\Leftrightarrow$(2) is proved in Theorem \ref{charinj1}. (2)$\Rightarrow$(3) is immediate, whereas (3)$\Rightarrow$(1) follows by Theorem \ref{injsubsp} and by the fact that $X$ can be embedded in a suitable $\ell_\infty(\Gamma)$, which is an injective Banach space.

(2)$\Rightarrow$(4) follows since (2) implies that $X$ is an ideal in any Banach space containing it, so in particular $X$ is an $L_1$ predual by \cite[Proposition 3.4]{fakh72}.

(4)$\Rightarrow$(3) also follows from Theorem \ref{injsubsp} since $X^{**}$ is an injective Banach space by the assumption that $X$ is an $L_1$ predual \cite[Theorem 6.1]{linds64}. 
\end{proof}

Now we can find examples of pairs of spaces $Y\subseteq X$ such that $Y$ is a $\kappa$ ideal in $X$ but $Y$ is not an $\alpha$ ideal in $X$ for any cardinal $\alpha>\kappa > \aleph_0$ (and henceforth such that $Y$ is not $1$-complemented in $X$).

\begin{example} Let $\kappa<\alpha$ be two infinite uncountable cardinals and let $X$ be a $\kappa$ injective Banach space which is not $\alpha$ injective. Then:
\begin{enumerate}
    \item $X$ is a $\kappa$ ideal in $X^{**}$.
    \item Given any set $\Gamma$ such that $X\subseteq \ell_\infty(\Gamma)$ isometrically, $X$ is a $\kappa$ ideal in $\ell_\infty(\Gamma)$.
\end{enumerate}
\end{example}

We end the section by providing several examples which pursue to find examples of $\kappa$ ideals out of the theory of $\kappa$ injective Banach spaces. Indeed, we aim to find such examples out of the class of $L_1$ preduals. To begin with, we need to go to tensor product spaces. We refer the reader to Subsection~\ref{subsc:tensorproduct} for the necessary notation.

\begin{theorem}\label{tensor}
Let $X$ and $Y$ be Banach spaces, $\kappa$ an uncountable cardinal and $E \subset X$, $F \subset Y$ $\kappa$ ideals in $X$ and $Y$ respectively. Then $E \tensor F$ is a $\kappa$ ideal in $X \tensor Y$, where $\tensor$ represents the injective or projective tensor product.
\end{theorem}

Compare the above result with the fact that if $Y$ is an ideal in $Z$ and $X$ is any Banach space then $Y\projtensor X$ (resp. $Y\injtensor X$) is an ideal in $Z\projtensor X$ (resp. $Z\injtensor X$) (c.f. e.g. \cite[Theorem 1 and Lemma 2]{rao01}).

\begin{proof}
We will give the proof for the projective tensor product as the proof for the injective case is analogous and even simpler because this tensor product respects subspaces.

As $E$ and $F$ are $\kappa$ ideals in $X$ and $Y$ respectively, in particular they are ideals, so $E \projtensor F$ is indeed a subspace of $X \projtensor Y$ \cite[Theorem 1]{rao01}. To see that it is a $\kappa$ ideal in $X \projtensor Y$ we must prove that given any subspace $W$ of $X \projtensor Y$ with $\dens (W) < \kappa$ there exists a norm $1$ operator $P: W \ra E \projtensor F$ which fixes points of $W \cap E \projtensor F$.

Since $\dens(W \cap E \projtensor F) < \kappa$  there exists a dense set $\{u_\alpha\}_{\alpha \in A} \subset W \cap E \projtensor F$ with $|A| < \kappa$, so for any $\alpha \in A$ we can obtain a sequence
$$\left\{ \sum_{i=1}^{n_{k}^\alpha} e_{ik}^\alpha \otimes f_{ik}^\alpha  \right\}_{k=1}^\infty \subset E \otimes F$$
that converges to $u_\alpha$. Furthermore, as $\dens(W) < \kappa$ we can also get a dense set $\{w_\beta\}_{\beta \in B} \subset W$ with $|B| < \kappa$ such that for any $\beta \in B$ there is a convergent sequence
$$\left\{ \sum_{i=1}^{m_{k}^\beta} x_{ik}^\beta \otimes y_{ik}^\beta  \right\}_{k=1}^\infty \subset X \otimes Y$$
whose limite is $w_\beta$. 

By an application of Theorem~\ref{theo:simyost} find ideals $U \subset X$ and $V \subset Y$ of density $<\kappa$ and such that 
\begin{align*}
 & U \supseteq \clspan \left\{e_{ik}^\alpha, x_{jk}^\beta: \alpha \in A \wedge \beta \in B \land k \in \N \land 1 \le i \le n_k^\alpha \land 1 \le j \le m_k^\beta\right\}, \\
 & V \supseteq \clspan \left\{f_{ik}^\alpha, y_{ik}^\beta: \alpha \in A \land \beta \in B \land k \in \N \land 1 \le i \le n_k^\alpha \land 1 \le j \le m_k^\beta\right\}.
\end{align*}
Observe that a density argument yields us that $W \subset U \projtensor V$. Moreover, since $U$ and $V$ are ideals in $X$ and $Y$ respectively we have that $U \projtensor V$ is a subspace of $X \projtensor Y$.

Now we use again that $E, F$ are  $\kappa$ ideals in $X$ and $Y$ respectively, together with the fact that $\dens(U), \dens(V) < \kappa$, to obtain two norm $1$ maps $\phi: U \ra E$ and $\psi: V \ra F$ that fix points of $U \cap E$ and $V \cap F$ respectively. That allows us to define $Q = \phi \otimes \psi: U \projtensor V \ra E \projtensor F$ and finally $P= Q|_W$ will be the desired projection.

Indeed, we have that
$$\norm{P} \le \norm{Q} = \norm{\phi} \norm{\psi} = 1$$
so, to conlcude, it is enough to see that $P(w) = w$ for every $w \in W \cap E \projtensor F$.

Given $w \in W \cap E \projtensor F$ and $\eps > 0$, we know that there exists some $\alpha \in A$ and $n_k^\alpha \in \N$ such that
\begin{align}\label{tensor5}
\norm{w - \sum_{i=1}^{n_{k}^\alpha} e_{ik}^\alpha \otimes f_{ik}^\alpha} < \varepsilon/2.
\end{align}
Furthermore, as every $e_{ik}^\alpha \in U \cap E$ and $f_{ik}^\alpha \in V \cap F$ we have that
\begin{align}\label{tensor6}
Q \left( \sum_{i=1}^{n_{k}^\alpha} e_{ik}^\alpha \otimes f_{ik}^\alpha \right) = \sum_{i=1}^{n_{k}^\alpha} Q(e_{ik}^\alpha \otimes f_{ik}^\alpha) = \sum_{i=1}^{n_{k}^\alpha} \phi(e_{ik}^\alpha) \otimes \psi(f_{ik}^\alpha) = \sum_{i=1}^{n_{k}^\alpha} e_{ik}^\alpha \otimes f_{ik}^\alpha.
\end{align}
By (\ref{tensor5}) and (\ref{tensor6}) we conclude that
\begin{align*}
\norm{P(w) - w}  &= \norm{Q(w) - w} \le \norm{Q(w) -Q\left( \sum_{i=1}^{n_{k}^\alpha} e_{ik}^\alpha \otimes f_{ik}^\alpha \right)} + \\
& +\norm{\sum_{i=1}^{n_{k}^\alpha} e_{ik}^\alpha \otimes f_{ik}^\alpha - w} \le \norm{Q} \norm{w - \sum_{i=1}^{n_{k}^\alpha} e_{ik}^\alpha \otimes f_{ik}^\alpha} + \\
 &  + \norm{w - \sum_{i=1}^{n_{k}^\alpha} e_{ik}^\alpha \otimes f_{ik}^\alpha} \le 2 \norm{w - \sum_{i=1}^{n_{k}^\alpha} e_{ik}^\alpha \otimes f_{ik}^\alpha} < \varepsilon.
\end{align*}
Since $\varepsilon$ was arbitrary we conclude that $P(w) = w$, as wanted.
\end{proof}

\begin{example}\label{exam:contratensor}
Let $X$ be a non-Asplund and $\kappa$ injective Banach space (e.g. $\ell_\infty(\kappa,\Gamma)$ for $\kappa<\vert\Gamma\vert$) and let $\ell_\infty(\Gamma)$ a space such that $X$ embeds isometrically in $\ell_\infty(\Gamma)$. Then $X\projtensor X$ is a $\kappa$ ideal in $\ell_\infty(\Gamma)\projtensor \ell_\infty(\Gamma)$ by a combination of Theorems~\ref{charinj1} and \ref{tensor}. Moreover, we claim that $X\projtensor X$ is not an $L_1$-predual space. 

Since $X$ is not Asplund but it is an $L_1$ predual, we have that $X^*=L_1(\mu)$ for some measure $\mu$ not purely atomic. Consequently, $X^*$ contains an isomorphic copy of $\ell_2$ and $(X\projtensor X)^*=L(X,X^*)\supset X^*\injtensor X^*$ contains $\ell_2\injtensor \ell_2$ isomorphically, which in turn contains an isomorphic copy of $c_0$ (this follows since $\ell_2\injtensor \ell_2$ is an $M$-ideal in its bidual \cite[Theorem VI.4.1]{hww} which is non-reflexive \cite[Theorem 4.21]{ryan}, so the result follows by \cite[Theorem II.4.7]{hww}). To conclude that $X\pten X$ is not an $L_1$ predual it remains to observe that no $L_1$ space may contain an isomorphic copy of $c_0$ (c.f. e.g. \cite[Theorem 1.c.4]{lita}).
\end{example}

In order to look for more examples of $\kappa$ ideals, we will focus our attention to \textit{Lipschitz-free spaces}. We refer the reader to Subsection~\ref{subsc:lipschitzfree} for the necessary notation.

Let us recall that as a consequence of results from \cite{fakh72} it follows that given a Banach space $X$ and a subspace $Y\subseteq X$ then $\mathcal F(Y)$ is an ideal in $\mathcal F(X)$ if, and only if, $Y$ is an ideal in $X$ (see the paragraph after Definition 1.1. in \cite{agmrz25} for a detailed explanation). In the transfinite case we can also prove one implication, which will be enough to expand our non $L_1$ predual examples, but we don't know if the reciprocal of the next theorem is true (see Question~\ref{question:lipfree}).

\begin{proposition}\label{theo:idealfree}
Let $N \subset M$ be Banach spaces such that $N$ is a $\kappa$ ideal in $M$ for some infinite uncountable cardinal $\kappa$. Then $\mathcal{F}(N)$ is a $\kappa$ ideal in $\mathcal{F}(M)$ through the canonical inclusion.
\end{proposition}

\begin{proof}
Let $E \subset \mathcal{F}(M)$ be a subspace with $\dens(E) < \kappa$, take a dense set $\{m_\alpha\}_{\alpha \in A} \subset E \cap \mathcal{F}(N)$ and extend it to a dense set $\{m_\alpha\}_{\alpha \in A} \cup \{e_\beta\}_{\beta \in B}$ of $E$ such that $|A \cup B| < \kappa$. Then, for every $\alpha \in A$ and $\beta \in B$ we can express
\begin{align*}
    m_\alpha & = \sum_{k=1}^\infty \lambda_\alpha^k m_{x_\alpha^k y_\alpha^k}, \\
    e_\beta & = \sum_{k=1}^\infty \mu_\beta^k m_{z_\beta^k w_\beta^k},
\end{align*}
for certain $\lambda_\alpha^k, \mu_\beta^k \in \R$, $x_\alpha^k, y_\alpha^k \in N$, $z_\beta^k, w_\beta^k \in M$, where $m_{x_\alpha^k y_\alpha^k}$ and $m_{z_\beta^k w_\beta^k}$ are molecules on $\mathcal{F}(M)$.

By defining 
$$K= \clspan\{x_\alpha^k, y_\alpha^k, z_\alpha^k, w_\alpha^k: \alpha \in A, \beta \in B, k \in \N\},$$
we have a subspace of $M$ with $\dens(K) < \kappa$, so by the assumption we obtain a norm $1$ operator $Q: K \ra N$ that acts as the identity on $K \cap N$. By the linerarisation property that defines the Lipschitz-free space, we can obtain a bounded operator $T_Q: \mathcal{F}(K) \ra \mathcal{F}(N)$ such that $\norm{T_Q} = \norm{Q} \le 1$ and $T_Q \circ \delta_K = \delta_N \circ Q$. Furthermore, as $E \subset \clspan \delta_M(K)$, if we consider $\iota: E \ra \mathcal{F}(K)$ the canonical inclusion we can define $P: E \ra \mathcal{F}(N)$ as $P = T_Q \circ \iota$ and it is enough to see that $P(e) = e$ for every $e \in E \cap \mathcal{F}(N)$.

To prove it, first observe that as $Q(x_\alpha^k) = x_\alpha^k$ we have that
$$T_Q(\delta_K(x_\alpha^k)) = \delta_N(Q(x_\alpha^k)) = \delta_N(x_\alpha^k).$$
Observe that we get an analogous equality for $y_\alpha^k$. Consequently
\begin{align}\label{Pmolec}
P(m_\alpha) = & T_Q\left(\sum_{k=1}^\infty \lambda_\alpha^k \iota\left(m_{x_n^ky_n^k}\right) \right) = T_Q \left(\sum_{k=1}^\infty \lambda_n^k \dfrac{\delta_K(x_\alpha^k) - \delta_K(y_\alpha^k)}{\norm{x_\alpha^k -y_\alpha^k}} \right) \nonumber\\
 = & \sum_{k=1}^\infty \lambda_\alpha^k \dfrac{\delta_N(x_\alpha^k) - \delta_N(y_\alpha^k)}{\norm{x_\alpha^k - y_\alpha^k}} = m_\alpha.
\end{align}
Then, given $e \in E \cap \mathcal{F}(N)$ and $\varepsilon >0$ we can take $m_\alpha$ such that
$$\norm{e - m_\alpha} < \varepsilon/2$$
and by (\ref{Pmolec})
\begin{align*}
\norm{P(e) - e} & \le \norm{P(e) - P(m_\alpha)} + \norm{P(m_\alpha) - e} \le (\norm{P} + 1) \norm{e - m_\alpha} \\
& \le 2 \norm{e - m_\alpha} < \varepsilon.   
\end{align*}
As the last inequality is true for any $\varepsilon > 0$, we can conclude that $P(e) = e$ as wanted.
\end{proof}

Observe once again that the above theorem produces examples of Banach spaces such that $\mathcal F(Y)$ is a $\kappa$ ideal in $\mathcal F(X)$ and which do not belong to the class of Banach spaces which are $L_1$ preduals. Indeed, any infinite dimensional Lipschitz free space is never isomorphic to an $L_1$ predual space \cite[Theorem 1.1 (iii)]{cdw}.

\section{Transfinite (almost) isometric ideals}\label{sect:transaiideal}

We start with the formal definition of $\kappa$ (almost) isometric ideal.

\begin{definition}\label{defi:transaiideal}
Let $X$ be a Banach space, $Y$ a subspace of $X$ and $\kappa$ an infinite cardinal. We say that $Y$ is a \textbf{$\kappa$} (almost) isometric ideal in $X$ if given any subspace $E$ of $X$ with $\dens(E) <\kappa$ (and $\eps >0$), there exists an ($\eps$-)isometry $T: E \ra Y$ that preserves the points of $E \cap Y$.
\end{definition}

It was proven in \cite[Theorem 4.3]{aln2} that a Banach space $X$ is a Gurari\u{\i} space if, and only if, $X$ is an almost isometric ideal in every Banach space containing it. Observing then that $\aleph_0$ (almost) isometric ideals are the classical (almost) isometric ideals (see Definition \ref{defideal} (2)), the following theorem extends this result to any space of (almost) universal disposition for spaces of density $< \kappa$:

\begin{theorem}\label{AUDchar}
Let $X$ be a Banach space and $\kappa$ an infinite cardinal. The following are equivalent:
\begin{enumerate}
    \item $X$ is a $\kappa$ (almost) isometric ideal in every Banach space that contains it.
    \item $X$ is a $\text{\emph{(A)UD}}_{<\kappa}$ space, that is, a space of (almost) universal disposition for Banach spaces of density $<\kappa$.
    \end{enumerate}
\end{theorem}

In the proof we will need to consider an important category-theoretic construction in Banach spaces known as the Push-Out construction. For completeness we state here a lemma with this construction:

\begin{lemma}\label{pushoutconst}
Let $A, B, Y$ be Banach spaces and $\alpha: Y \ra A$, $\beta: Y \ra B$ bounded linear operators. Then, there exists a Banach space $PO$ (which is known as \emph{push-out}) and two bounded linear operators $\alpha':B\ra PO$ and $\beta':A\ra PO$ that makes the following diagram commute:
\begin{center}
\begin{tikzcd}
    Y \arrow{d}{\beta} \arrow{r}{\alpha} & A \arrow{d}{\beta'} \\
    B \arrow{r}{\alpha'} & PO
\end{tikzcd}
\end{center}
Furthermore, if $\alpha$ and $\beta$ are isometries, then so are $\alpha', \beta'$.
\end{lemma}

In essence, the push-out can be constructed as the quotient of the sum $A \oplus_1B$ with the closure of the subspace $\Delta = \{(\alpha(y), - \beta(y)): y \in Y\}$. The details can be found in \cite[Lemma A.19]{sepibook}.

\begin{proof}
Let us begin by proving that (1) implies (2).
    
Given $A \subset B$ normed spaces with $\dens(B) < \kappa$ and $t: A \ra X$ an isometry, we must see that there exists some extension $T: B \ra X$  which is an ($\eps$-)isometry.

Let us consider $\overline{B}$ the completion of $B$ and $\overline{A}$ the closure of $A$ on $\overline{B}$. Then $\overline{A} \subset \overline{B}$ are Banach spaces and we can denote as $\iota: \overline{A} \ra \overline{B}$ the inclusion. Furthermore, we can extend $t$ to a unique linear isometry $\overline{t}: \overline{A} \ra X$. By Lemma \ref{pushoutconst} we get that $\overline{t}', \iota'$ in the following diagram
\begin{center}
\begin{tikzcd}
    \overline{A} \arrow{d}{\iota} \arrow{r}{\overline{t}} & X \arrow{d}{\iota'} \\
    \overline{B} \arrow{r}{\overline{t}'} & PO
\end{tikzcd}
\end{center}
are also isometries, so $X \equiv \iota'(X)$ and by (1) (being a $\kappa$ (almost) isometric ideal in every Banach space is a property preserved by isometries) we obtain that $\iota'(X)$ is a $\kappa$ (almost) isometric ideal in $PO$. Furthermore, as we also have that $B \equiv \overline{t}'(B)$, so $\dens(\overline{t}'(B)) < \kappa$, we can obtain an ($\eps$-)isometry $T': \overline{t}'(B) \ra \iota'(X)$ such that
\begin{align}\label{T'prese}
    T'(x) = x, \, \forall x \in \overline{t}'(B) \cap \iota'(X).
\end{align}
It is then enough to define $T = (\iota')^{-1} \circ T' \circ \overline{t}' \circ j$ (with $j: B \ra \overline{B}$ the inclusion), which is clearly an ($\eps$-)isometry, and observe that given $a \in A$ we have that
$$\overline{t}'(j(a)) = \overline{t}'(a) = \overline{t}'(\iota(a)) = \iota'(\overline{t}(a)) \in \overline{t}'(B) \cap \iota'(X),$$
so by (\ref{T'prese})
\begin{align*}
T(a) & =  (\iota')^{-1} \left(T'\left(\overline{t}'(j (a)\right)\right) = (\iota')^{-1} \left(T'\left(\iota'(\overline{t}(a))\right)\right) = (\iota')^{-1} \left(\iota'\left(\overline{t}(a)\right)\right) \\
 & = \overline{t}(a) = t (a)
\end{align*}
as desired.

The proof of (2) $\Rightarrow$ (1) is simpler.

Given $Z$ a Banach space that contains $X$, $E \subset Z$ with $\dens(E) < \kappa$ (and $\eps >0$), we can consider $\iota:E \cap X \ra X$ the inclusion and by (2) we obtain $T: E \ra X$ an extension of $\iota$ that is an ($\eps$-)isometry. Then, given any $e \in E \cap X$ we have that
$$T(e) =  \iota (e) = e,$$
so $T$ is the operator which proves that $X$ is a $\kappa$ (almost) isometric ideal in the space $Z$.
\end{proof}

As well as happen with the interrelation between of being an $L_1$ predual and the notion of ideal, the interrelation between Gurari\u{\i} spaces and the notion of almost isometric ideal goes beyond to the results of \cite{aln2}. Indeed, in \cite[Proposition 1]{rao16} it is proved that  if $X$ is a Gurari\u{\i} space and $Y$ is a closed subspace, then $Y$ is a Gurari\u{\i} space if, and only if, $Y$ is an almost isometric ideal in $X$. Next, we obtain a generalization of this result for spaces of $\kappa$ (almost) universal disposition and the notion of $\kappa$ (almost) isometric ideals:

\begin{theorem}\label{AUDsubsp}
Let $X$ be a $\text{\emph{(A)UD}}_{<\kappa}$ space and $Y$ a closed subspace of $X$. Then, $Y$ is a $\text{\emph{(A)UD}}_{<\kappa}$ space if, and only if, $Y$ is a $\kappa$ (almost) isometric ideal in $X$.
\end{theorem}

\begin{proof}
If $Y$ is a $\text{(A)UD}_{<\kappa}$ space, then it is a $\kappa$ (almost) isometric ideal in $X$ by Theorem~\ref{AUDchar}, so it is enough to prove the reciprocal statement.

Given $A \subset B$ normed spaces with $\dens(B) < \kappa$, an isometry $t: A \ra Y$ (and $\eps > 0$), we can define the isometry $t': A \ra X$ as $t' = \iota \circ t$ (where $\iota: Y \ra X$ stands for the inclusion operator) and since $X$ is a (A)UD$_{< \kappa}$ space we obtain a ($\delta$-)isometry $T': B \ra X$ (with $\delta < \min\{\sqrt{1+\eps}-1,1- \sqrt{1-\eps}\}$) that extends $t'$. Then, by defining $B' = T'(B) \subset X$ we have that $\dens(B') \le \dens(B) < \kappa$, so there exists some ($\delta$)-isometry $T'': B' \ra Y$ which preserves points of $B' \cap Y$. It remains to define $T = T'' \circ T'$, which will be the ($\eps$-)isometry we were looking for.
\end{proof}

Bearing the above results in mind we get the following corollary, which sums up a characterisation of $\text{(A)UD}_{<\kappa}$ Banach spaces.

\begin{corollary}
Let $X$ be a Banach space. The following assertions are equivalent:
\begin{enumerate}
    \item $X$ is a $\text{\emph{(A)UD}}_{<\kappa}$ space.
    \item $X$ is a $\kappa$ (almost) isometric ideal in every Banach space $X$ containing it.
    \item $X$ is a $\kappa$ (almost) isometric ideal in every $\text{\emph{(A)UD}}_{<\kappa}$ Banach space containing it.
\end{enumerate}
\end{corollary}

\begin{proof}
(1)$\Leftrightarrow$(2) is proved in Theorem \ref{AUDchar}. (2)$\Rightarrow$(3) is immediate, whereas (3)$\Rightarrow$(1) follows by Theorem \ref{AUDsubsp} and the fact that every Banach space is contained in some $\text{(A)UD}_{<\kappa}$ Banach space \cite[Theorem 4.2]{gk11}. 
\end{proof}

With the aid of the characterisations of $\text{{UD}}_{<\kappa}$ spaces we will show examples that justify that the notions of $\kappa$ almost isometric ideal we have introduced are different from each other for different cardinals (compare with Example \ref{ex:distintai}).

\begin{example}\label{remark:ejemplosdisuninosuperior}
By \cite[Theorem 4.4]{gk11}, given any uncountable cardinal $\kappa$ with $\kappa^{<\kappa}=\kappa$ there exists a (unique) Banach space $X$ with $\dens(X)=\kappa$ and of universal disposition for Banach spaces of density $<\kappa$. Then, for any Banach space $Y$ with density greater than $2^\kappa$ such that $X \subset Y$, by Theorem~\ref{AUDchar} we get that $X$ is a $\kappa$ isometric ideal in $Y$, but it cannot be an $\alpha$ almost isometric ideal in $Y$ for any $\kappa^+ < \alpha$, because in that case we would have an isomorphism of some space of density grater than $\kappa$ into $X$. 
\end{example}

Now it is time for more examples of $\kappa$ (almost) isometric ideals.

\begin{proposition}
Let $\kappa$ be an infinite uncountable cardinal and $J \subset I$ sets such that $\kappa \le |J| < |I|$. Then $\ell_p(J)$ is a $\kappa$ isometric ideal in $\ell_p(I)$ for $1 \le p < \infty$. The same is true for $c_0(J)$ and $c_0(I)$.
\end{proposition}

\begin{proof}
We will only provide the proof for the case of $\ell_p(I)$  as the one for $c_0(I)$ is analogous.

Let $E \subset \ell_p(I)$ be a subspace with density $< \kappa$ and take a dense subset $\{e_\alpha\}_{\alpha \in \Lambda}$ of $E$ with $|\Lambda| < \kappa$. Define
$$A = \bigcup_{\alpha \in \Lambda} \sop(e_\alpha)$$
and observe that $\abs{A} < \kappa\le \vert J\vert$ and, furthermore, that given any $e \in E$ we have that $\sop(e) \subset A$. We can then define an injective map $\phi: A \ra J$ that fixes the points of $A \cap J$ and an operator $T = S \circ \psi \circ R: E \ra \ell_p(J)$ where:
\begin{itemize}
    \item $R: E \ra \ell_p(A)$ is the restriction map over $A$, so it will be an isometry because the support of every element of $E$ is on $A$.
    \item $\psi: \ell_p(A) \ra \ell_p(\phi(A))$ is the isometry given by $\psi(x) = x \circ \phi^{-1}$.
    \item $S: \ell_p(\phi(A)) \ra \ell_p(J)$ is the canonical inclusion.
\end{itemize}
It is then obvious that $T$ is a linear isometry and we just need to show that it fixes the points of $E \cap \ell_p(J)$. To see that, first observe  that given $e \in E \cap \ell_p(J)$ we have that
$$\sop(e) \subset A \cap J \subset \phi(A).$$
We can now distinguish three cases:
\begin{itemize}
    \item If $i \not \in \phi(A)$, then $e(i) = 0$ and obviously
    $$T(e)(i) = S(\psi(R(e)))(i) = 0 = e(i).$$
    \item If $i \in \phi(A) \setminus \sop(e)$, then $i = \phi(a)$ for some $a \in A \setminus \sop(e)$ (as $\sop(e) \subset A \cap J$ we get that $\phi(\sop(e)) = \sop(e)$) and so we have that
    \begin{align*}
        T(e)(i) & = S(\psi(R(e)))(i) = \psi(R(e))(i) = \psi(e|_A)(i) = e(\phi^{-1}(i)) \\
        & = e(a) = 0 = e(i).
    \end{align*}
    \item If $i \in \sop(e)$, then $i \in A \cap J \subset \phi(A)$ and $\phi(i) = i$, so $\phi^{-1}(i) = i$ and we get that
    $$T(e)(i) = S(\psi(R(e)))(i) = \psi(R(e))(i) = e(\phi^{-1}(i)) = e(i).$$
\end{itemize}
These three cases show that $T(e) = e$ when $e \in E \cap \ell_p(J)$ as we wanted to prove.
\end{proof}

\begin{remark}
Observe that the condition that $\kappa\le \vert J\vert$ can not be removed in the assumptions of the above result.

Indeed, if we assume $\vert J\vert<\kappa$, we can take $i_0 \in I \setminus J$ and consider the subspace
$$E = \clspan\{e_j: j \in J \cup \{i_0\}\} = \ell_p(J \cup\{i_0\})$$
which has density $< \kappa$. Then, is is clear that there does not exist an isomorphism (into its image) $T: E \ra \ell_p(J)$ that fixes the elements of $E \cap \ell_p(J) = \ell_p(J)$.
\end{remark}

Next we will show that the injective tensor product preserves $\kappa$ (almost) isometric ideals, which will give us another family of examples. In the proof it is essential that this tensor product respects isometries, which is the main reason behind the exclusion of the case of the projective tensor product (compare the following result with the establishment of Theorem~\ref{tensor}). Indeed, we will show in Example~\ref{exam:contraproyect} that such result is false for the projective tensor product.

\begin{theorem}\label{theo:injetensortransai}
Let $X$ and $Y$ be Banach spaces, $\kappa$ an infinite cardinal and $E \subset X$, $F \subset Y$ two $\kappa$ (almost) isometric ideals. Then, $E \injtensor F$ is a $\kappa$ (almost) isometric ideal in $X \injtensor Y$.
\end{theorem}

When $\kappa$ is uncountable the proof is completely analogous to the one given for Theorem \ref{tensor}. The only change here is that we use the well known fact that if $\phi, \psi$ are ($\delta$-)isometries, then $\phi \injtensor \psi$ is also an ($\eps$-)isometry (where we choose $\delta$ as in the proof of Theorem \ref{AUDsubsp}). The proof when $\kappa = \aleph_0$, that is, for the classical notion of (almost) isometric ideal will be given in Theorem~\ref{theorem:aiidealinyect} because this case will make use of arguments that only work for finite dimensional spaces.

We end the section by establishing connections with recent transfinite versions of octahedral norms and of almost square Banach spaces. Let us borrow the following definition from \cite{acllr23}.

\begin{definition}\label{defi:asqyohlargos} Let $X$ be a Banach space and $\kappa$ an infinite cardinal. 
\begin{itemize}
	    \item[(a)] We say that $X$ is \emph{$<\kappa$-octahedral} if, for every subspace $Y\subset X$ with $\dens(Y)<\kappa$ and $\varepsilon>0$, there exists $x\in S_X$ such that
        $$\|y+\lambda x\|\ge(1-\varepsilon)(\|y\|+|\lambda|)$$
        holds for every $y\in Y$ and every $\lambda\in\mathbb R$.
	    \item[(b)] We say that $X$ fails the \textit{$(-1)$-BCP$_{<\kappa}$} if , for every subspace $Y\subset X$ with $\dens(Y)<\kappa$, there exists $x\in S_X$ such that
        $$\|y+\lambda x\|=\|y\|+|\lambda|$$
        holds for every $y\in Y$ and every $\lambda\in\mathbb R$.
        \item[(c)] We say that $X$ is \textit{$<\kappa$-almost square} ($\ASQ{\kappa}$, for short) if, for every subspace $Y\subset X$ with $\dens(Y)<\kappa$ and $\varepsilon>0$, there exists $x\in S_X$ such that $$\Vert y+\lambda x\Vert\le (1+\varepsilon)\max\{\Vert y\Vert,\vert \lambda\vert\}$$
        holds for every $y\in Y$ and every $\lambda\in\mathbb R$.
	    \item[(d)] We say that $X$ is \textit{$<\kappa$-square} (${SQ}_{<\kappa}$, for short) if, for every subspace $Y\subset X$ with $\dens(Y)<\kappa$, there exists $x\in S_X$ such that $$\Vert y+\lambda x\Vert\le \max\{\Vert y\Vert,\vert \lambda\vert\}$$
    holds for every $y\in Y$ and every $\lambda\in\mathbb R$.
\end{itemize}
\end{definition}

$<\kappa$-octahedral spaces and spaces failing $(-1)$-BCP$_{<\kappa}$ were first introduced in \cite{cll23} (we took notation from \cite{acllr23}, which is a little different in order to avoid technical problems when dealing with limit cardinals). Anyway, the authors introduced this notion in order to provide a natural transfinite version of the classical notion of \textit{octahedral Banach spaces} (see. e.g. \cite[Chapters 3 and 5]{lan} for background on octahedral spaces). We refer to \cite{agmrz23,cll23} for background about $<\kappa$-octahedral spaces and spaces failing $(-1)$-BCP$_{<\kappa}$. 

Concerning $\ASQ{\kappa}$ and ${SQ}_{<\kappa}$, these notions were introduced in \cite{acllr23} in order to provide a natural transfinite version of \textit{almost square spaces} (see \cite{all16,blr16} for background on almost square Banach spaces).

The first aim is to show that transfinite octahedrality and transfinite almost squareness are properties inherited by transfinite almost isometric ideals (similarly, transfinite squareness and the failure of  $(-1)$-BCP are inherited by transfinite isometric ideals). Before the formal presentation of the results we will enunciate the following lemma of geometric nature, which will be used in the following result, for easy reference.

\begin{lemma}\label{lemma:geompropilargas}\cite[Lemma 2.2]{all16}
Let $X$ be a Banach space, $x,y\in S_X$ and $\varepsilon>0$. If $\Vert x\pm y\Vert\le 1+\varepsilon$, then
    $$(1-\varepsilon)\max\{\vert \alpha\vert,\vert \beta\vert\}\leq \Vert \alpha x+\beta y\Vert \leq (1+\varepsilon)\max\{\vert \alpha\vert,\vert \beta\vert\}.$$
\end{lemma}

Now we are ready to provide the following result.

\begin{proposition}\label{teo:hereaiidealesasqlargos}
Let $X$ be an $\ASQ{\kappa}$ Banach space and $Y\subseteq X$ a $\kappa$ almost isometric ideal in $X$. Then $Y$ is $\ASQ{\kappa}$.
\end{proposition}

\begin{proof}
Let $Z\subseteq Y$ be a subspace with $\dens(Z)<\kappa$ and let $\varepsilon>0$. Let us find $y\in S_Y$ such that $\Vert z+\lambda y\Vert\leq (1+\varepsilon)\max\{\Vert z\Vert,\vert\lambda\vert\}$ for any $z \in Z$ and $\lambda \in \R$. Since $Z\subseteq X$ we can find $x\in S_X$ satisfying that
$$\Vert z+\lambda x\Vert\leq \left(1+\eps\right)\max\{\Vert z\Vert,\vert\lambda\vert\}.$$
Define $E:=\lspan\{Z\cup\{x\}\}$, which is a subspace of $X$ such that $\dens(E)<\kappa$. By the assumption there exists $\phi: E\longrightarrow Y$ satisfying that $\phi(e)=e$ holds for every $e\in Y\cap E$ and such that, for every $e\in E$ the following inequality holds 
$$\left(1-\eps\right)\Vert e\Vert\leq \Vert \phi(e)\Vert\leq \left(1+\eps\right)\Vert e\Vert.$$
Let $y:=\frac{\phi(x)}{\Vert \phi(x)\Vert}\in S_Y$, and let us show that $y$ satisfies our purposes. 
Observe that
$$\Vert y-\phi(x)\Vert=\vert 1-\Vert \phi(x)\Vert \vert\leq\eps$$
and so
\[\begin{split}
\left\Vert \frac{z}{\Vert z\Vert}\pm y\right\Vert&\leq \left\Vert \frac{z}{\Vert z\Vert}\pm \phi(x)\right\Vert+\Vert y-\phi(x)\Vert \leq \left\Vert \phi\left(\frac{z}{\Vert z\Vert}\pm x\right)\right\Vert+\eps\\
& \leq (1+\varepsilon) \left\Vert \frac{z}{\Vert z\Vert}\pm x\right\Vert+\eps\leq (1+\varepsilon)^2+\eps\leq 1+4\varepsilon.\end{split}
\]
Now Lemma~\ref{lemma:geompropilargas} implies, 
$$(1-4\varepsilon)\max\{\Vert z\Vert,\vert\lambda\vert\}\leq \left\Vert \Vert z\Vert \frac{z}{\Vert z\Vert}+\lambda y\right\Vert=\Vert z+\lambda y\Vert \leq (1+4\varepsilon)\max\{\Vert z\Vert, \vert \lambda\vert\}.$$
Summarising we have proved that for every  $Z\subseteq Y$ such that $\dens(Z)<\kappa$ and every $\varepsilon>0$ there exists $y\in S_Y$ such that $\Vert z+\lambda y\Vert\leq (1+4\varepsilon)\max\{\Vert y\Vert,\vert \lambda\vert\}$ holds for every $z\in Z$ and every $\lambda\in\mathbb R$. The arbitrariness of $\varepsilon>0$ implies that $Y$ is $\ASQ{\kappa}$, as desired.
\end{proof}

An adaptation of the proof above allows us to obtain a rigid version in the following sense.

\begin{proposition}\label{prop:heredaiisosq}
Let $X$ be an $\SQ{\kappa}$ Banach space and $Y\subseteq X$ a $\kappa$ isometric ideal in $X$. Then $Y$ is $\SQ{\kappa}$.
\end{proposition}

Moreover, we can obtain versions for $< \kappa$-octahedral spaces with similar proofs.

\begin{proposition}\label{bajaocta}
Let $X$ be a Banach space and let $Y$ be a subspace.
\begin{enumerate}
    \item Assume that $X$ is $<\kappa$-octahedral and that $Y\subseteq X$ is a $\kappa$ almost isometric ideal in $X$. Then $Y$ is $<\kappa$-octahedral.
    \item Assume that  $X$ fails the $(-1)$-BCP$_{<\kappa}$ and that $Y\subseteq X$ is a $\kappa$ isometric ideal in $X$. Then $Y$ fails $(-1)$-BCP$_{<\kappa}$.
\end{enumerate}
\end{proposition}

In the sequel we will prove that, indeed, the above four properties have a characterisation in terms of transfinite (almost) isometric ideals.

\begin{theorem}\label{theo:caraasqlargosaiideales}
Let $X$ be a Banach space and $\kappa$ an infinite cardinal. The following are equivalent:
\begin{enumerate}
    \item $X$ is $\ASQ{\kappa}$.
    \item $X\times \{0\}$ is a $\kappa$ almost isometric ideal in $X\oplus_\infty\mathbb R$.
\end{enumerate}
\end{theorem}

\begin{proof}
1$\Rightarrow$2. Let $E\subseteq X\oplus_\infty\mathbb R$ be a subspace such that  $\dens(E) <\kappa$ and let $\varepsilon>0$. Consider $P:X\oplus_\infty\mathbb R\longrightarrow X\times\{0\}$ the natural projection and define $Y:=P(E)\subseteq X\times \{0\}$. Since $Y$ satisfies $\dens(Y)<\kappa$ and $X\times \{0\}$ is $\ASQ{\kappa}$ (since the above space is clearly isometrically isomorphic to $X$) there exists $x\in S_X$ satisfying
\begin{align}\label{ASQ1}
\Vert (y+\lambda x,0)\Vert\leq (1+\varepsilon)\max\{\Vert y\Vert,\vert \lambda\vert\}
\end{align}
for every $(y, 0) \in Y$ and $\lambda \in \R$.

Define $T:E\longrightarrow X\times\{0\}$ by the equation
$$T(z,\lambda):=(z+\lambda x,0).$$
It is obvious that $T(z,0)=(z,0)$ holds for every $(z,0)\in E\cap X\times \{0\}$. On the other hand, given $(z,\lambda) \in E$ we have that $(z,0) \in Y$ and if $z\neq 0$ by (\ref{ASQ1}) we obtain that
$$\left\Vert \frac{z}{\Vert z\Vert}\pm x\right\Vert = \norm{\left(\frac{z}{\Vert z\Vert}\pm x, 0  \right)}\leq 1+\varepsilon.$$
By an application of Lemma~\ref{lemma:geompropilargas}, in virtue of the equality $\Vert (z+\lambda x,0)\Vert=\Vert \Vert z\Vert \frac{z}{\Vert z\Vert}+\lambda x\Vert$, we conclude that 
\[\begin{split}(1-\varepsilon)\Vert (z,\lambda)\Vert & =  (1-\varepsilon)\max\{\Vert z\Vert,\vert \lambda\vert\} \leq \Vert (z+\lambda x,0)\Vert = \norm{T(z,\lambda)}\\
& \leq (1+\varepsilon)\max\{\Vert z\Vert,\vert \lambda\vert\}=(1+\varepsilon)\Vert (z,\lambda)\Vert.\end{split}\]
Consequently (2) follows.

2$\Rightarrow$1. Let us prove that $X\times \{0\}$ is $\ASQ{\kappa}$. To this end, take a subspace $Y\subseteq X\times \{0\}$ such that $\dens(Y)<\kappa$ and $\varepsilon>0$. Define
$$E:=\lspan\{Y\cup\{(0,1)\}\}\subseteq X\oplus_\infty \mathbb R$$
which will have density $< \kappa$. By the assumption there exists some $T:E\longrightarrow X\times \{0\}$ such that $T(y,0)=(y,0)$ holds for every $(y,0)\in E$ and such that
\[\begin{split}(1-\varepsilon)\max\{\Vert y\Vert,\vert\lambda\vert\} =(1-\varepsilon)\Vert (y,\lambda)\Vert& \leq \Vert T(y,\lambda)\Vert\leq (1+\varepsilon)\Vert (y,\lambda)\Vert\\
& =(1+\varepsilon)\max\{\Vert y\Vert,\vert \lambda\vert\}.
\end{split}\]
In other words
$$(1-\varepsilon)\max\{\Vert y\Vert,\vert\lambda\vert\}\leq \Vert T(y,0)+\lambda T(0,1)\Vert\leq (1+\varepsilon)\max\{\Vert y\Vert,\vert \lambda\vert\}.$$
In particular, for $(y,0) \in Y$ and $\lambda \in \R$
\[\begin{split}
\left \Vert (y,0)+\lambda\frac{T(0,1)}{\Vert T(0,1)\Vert}\right\Vert& =\left\Vert T(y,0)+\frac{\lambda}{\Vert T(0,1)\Vert}T(0,1)\right\Vert \\
& 
 \leq (1+\varepsilon)\max\left\{\Vert y\Vert,\frac{\abs{\lambda}}{\Vert T(0,1)\Vert}\right\} \\
& =\frac{1+\varepsilon}{\Vert T(0,1)\Vert}\max\{\norm{T(0,1)}\Vert y\Vert,\vert\lambda\vert\}\\
& \leq \frac{1+\varepsilon}{\Vert T(0,1)\Vert}\max\{(1+\eps)\Vert y\Vert,(1+\eps)\vert\lambda\vert\}\\
& = \frac{(1+\varepsilon)^2}{1-\varepsilon}\max\{\Vert y\Vert,\vert\lambda\vert\}.
\end{split}\]
This finishes the proof.
\end{proof}

Similar arguments provide versions for $\SQ{\kappa}$ and $<\kappa$-octahedral spaces and for the failure of $(-1)$-BCP$_{<\kappa}$. We omit the proof of the following theorem to avoid repetitions.

\begin{theorem}\label{theo:caraasqlargosaiidealesbis}
Let $X$ be a Banach space and $Y\subseteq X$ a subspace. Then:
\begin{enumerate}
    \item $X$ is $\SQ{\kappa}$ if, and only if, $X\times \{0\}$ is a $\kappa$  isometric ideal in $X\oplus_\infty\mathbb R$.
    \item $X$ is $<\kappa$-octahedral if, and only if, $X\times \{0\}$ is a $\kappa$ almost isometric ideal in $X\oplus_1\mathbb R$.
    \item $X$ fails the $(-1)$-BCP$_{<\kappa}$ if, and only if, $X\times \{0\}$ is a $\kappa$ isometric ideal in $X\oplus_1\mathbb R$.
\end{enumerate}
\end{theorem}

Observe that by taking $\kappa = \aleph_0$ in the previous theorems we obtain the corresponding results for classical octahedrality and almost squareness, which is probably well known by specialist, but for which we have not found any explicit reference. We leave here the statement.

\begin{corollary}\label{coro:aiiealoctaclasisuma}
Let $X$ be a Banach space. Then:
\begin{enumerate}
    \item $X$ is octahedral if, and only if, $X\times\{0\}$ is an almost isometric ideal in $X\oplus_1 \mathbb R$.
    \item $X$ is ASQ if, and only if, $X\times\{0\}$ is an almost isometric ideal in $X\oplus_\infty \mathbb R$.
\end{enumerate}
\end{corollary}

We end the section by observing that an easy application of Theorems \ref{theo:caraasqlargosaiideales}, \ref{theo:caraasqlargosaiidealesbis} and \ref{AUDchar} allows us to prove that spaces of almost universal disposition enjoy transfinite versions of octahedrality and almost squareness. Similarly, spaces of universal disposition are square and fail $(-1)$-BCP. Observe that the version for (A)SQ reproves \cite[Example 2.6]{acllr23}. However, even though the result is probably well known for specialists, we have not found an explicit reference for the result about transfinite octahedrality and the failure of $(-1)$-BCP.

\begin{corollary}\label{cor:propUDoctaasq}
Let $X$ be a Banach space and $\kappa$ an infinite cardinal. Then:
\begin{enumerate}
    \item If $X$ is an $\text{\emph{AUD}}_{<\kappa}$ space, then $X$ is ASQ$_{<\kappa}$ and it is $<\kappa$-octahedral.
    \item If $X$ is a $\text{\emph{UD}}_{<\kappa}$ space, then $X$ is SQ$_{<\kappa}$ and fails $(-1)$-BCP$_{<\kappa}$.
\end{enumerate}
\end{corollary}

Let us conclude with an example which distinguish transfinite versions of almost isometric ideals and isometric ideals.

\begin{example}\label{ex:distintai}
Let $\kappa$ be an uncountable cardinal, $X:=c_0(\mathbb N\setminus\{1\}, \ell_n(\kappa))$ and $Y:=X\oplus_\infty\mathbb R$. By \cite[Example 3.1]{acllr23} we know that $X$ is an ASQ$_{<\kappa}$ space, but $X$ does not contain any isomorphic copy of $c_0(\aleph_1)$ and, in particular, $X$ cannot be equivalently renormed to be SQ$_{<\aleph_1}$. According to Theorems~\ref{theo:caraasqlargosaiideales} and \ref{theo:caraasqlargosaiidealesbis} we infer that $X\times\{0\}$ is a $\kappa$ almost isometric ideal in $Y$ but $X\times\{0\}$ fails to be even an $\aleph_1$ isometric ideal in $Y$.
\end{example}

\section{A revision of classical results}\label{sect:classicalresults}

Concerning the theory of ideals or almost isometric ideals we have two important theorems which show that these notions are abundant. On the one hand, the Principle of Local Reflexivity can be established as the fact that any Banach space is an almost isometric ideal in its bidual. On the other hand, a recent result by T. A. Abrahamsen \cite[Remark 2.3]{abrahamsen15} yields in particular that given any Banach space $X$ and any subspace $Y\subseteq X$ there exists an almost isometric ideal $Z$ in $X$ with $\dens(Z)=\dens(Y)$ and $Y\subseteq Z$ (a version for the notion of ideal is a classical result by B. Sims and D. Yost \cite{simyos}). The main aim of this section is to make an intensive analysis of possible extensions of the above results to the transfinite generalisations of ideals and of (almost) isometric ideals that we have introduced in this paper. 

\subsection{The Principle of Local Reflexivity}\label{subsect:principlelocal}

In this section we pursue to prove that no generalisation of the principle of local reflexivity is possible out of the case of almost isometric ideals.

As said, the essence of this principle is that any Banach space is an almost isometric ideal in its bidual and the next example shows that it will no longer be true for $\kappa$ almost isometric ideals if $\kappa$ is uncountable:

\begin{example}
If $\kappa$ is an uncountable cardinal, then the space $X= c_0(\kappa)$ cannot be a $\kappa$ almost isometric ideal in its bidual. Indeed, if that was the case, as $X^{**}$ contains an isometric copy of $\ell_1$, we would have an isomorphism (into its image) $P: \ell_1 \ra c_0(\kappa)$ which is not possible as $c_0(\kappa)$ is Asplund while $\ell_1$ is not.
\end{example}

Furthermore, the next example shows that any version of the principle of local reflexivity is false if we consider $\kappa$ (isometric) ideals even in the case where $\kappa = \aleph_0$.

\begin{example}\label{exam:plrrigidofalsofinito}
Consider $X=c$ the space of scalar convergent sequences and select any Banach space $Y$ and any finite dimensional Banach space such that the set of extreme points of $B_{Y^*}$ is not isolated. By \cite[Theorem 7.6]{linds64} there exists an operator $T:Y\longrightarrow X$ and a Banach space $Z\supset Y$ such that $Z/Y$ is $1$-dimensional and such that $T$ does not admit any norm-preserving extension to $Z$.

Since $X^{**}$ is an injective Banach space (for instance by \cite[Theorem 6.1]{linds64}), we claim that $X$ is not an $\aleph_0$ ideal in $X^{**}$. Indeed, if we assumed by contradiction that $X$ is an $\aleph_0$ ideal in $X^{**}$, consider $i:X\longrightarrow X^{**}$ the inclusion operator. By the injectivity we would have an operator $\hat T:Z\longrightarrow X^{**}$ such that $\Vert \hat T\Vert=\Vert i\circ T\Vert=\Vert T\Vert$ (since $i$ is an isometry) and such that $\hat T(y)=(i\circ T)(y)$ for any $y \in Y$. Since $Z$ is a finite dimensional space we would get that $\hat T(Z)\subseteq X^{**}$ is finite dimensional. By the assumption that $X$ is an $\aleph_0$ ideal in $X^{**}$ we could find an operator $P:\hat T(Z)\longrightarrow X$ with $\Vert P\Vert=1$ and $P(i(x))=x$ for every $x\in X$ such that $i(x)\in \hat T(Z)$. It follows that $P\circ\hat T:Z\longrightarrow X$ is an operator with $\Vert P\circ\hat T\Vert\leq \Vert T\Vert$ and such that, given $y\in Y$, we have
$$P(\hat T(y))=P(i(T(y)))=T(y),$$
so $P\circ \hat T$ is a norm preserving extension of the operator $T$, which yields the contradiction that we were looking for.
\end{example}

\subsection{Abrahamsen and Sims-Yost theorem}\label{subsect:simyostfalsos}

We will devote this subsection to an analysis of possible transfinite versions of Theorems~\ref{theo:simyost} and \ref{theo:abrahamsenai}. To be more precise, we will analyse the following question:
\begin{center}
Given $Y\subseteq X$ two Banach spaces, is there any $\kappa$ (almost isometric) ideal $Z$ in $X$ with $Y\subseteq Z\subseteq X$ and $\dens(Z)=\dens(Y)$?
\end{center}

Let us begin with the case $\kappa=\aleph_0$  for isometric ideals. The answer in this case, is negative.

\begin{example}\label{exam:noabrahamsenrigido}
Let $X=L_1([0,1])$. Observe that $X^{**}=L_\infty([0,1])^*$ fails the $(-1)$-BCP$_{<\aleph_0}$ since $X^{**}$ is the dual of a Banach space with the Daugavet property \cite[Lemma 5.1]{cll23}. It follows that there is no separable isometric ideal $Z$ in $X$ such that $X\subseteq Z\subseteq X^{**}$. Indeed, if $Z$ is any isometric ideal in $X^{**}$ it follows that given $x\in S_Z$, then there would be $z \in S_Z$ such that $\Vert x\pm z\Vert=2$ (by the failure of $X^{**}$ of the $(-1)$-BCP$_{<\aleph_0}$ and Theorem \ref{bajaocta} (2)). This implies that $x$ can not be a point of G\^ateaux differentiability. Since $x$ was arbitrary we conclude that the norm of $Z$ does not have any point of G\^ateaux differentiability and, consequently, $Z$ can not be separable (c.f. e.g. \cite[Theorem 8.14]{fab}). 
\end{example}

For the transfinite versions, we aim to find consistent counterexamples for any successor cardinal. In order to do so, we will analyse the case of $\kappa$ ideals in $\ell_\infty(\kappa)$ containing $c_0(\kappa)$. In that line we get the following results:

\begin{proposition}\label{prop:kappaidealc0}
Let $X$ be a $\kappa$ ideal in $\ell_\infty(\Gamma)$ which contains $c_0(\Gamma)$. Then, for every $f \in \ell_\infty(\kappa, \Gamma)$ there exists a function $g \in X$ such that $g(\gamma) = f(\gamma)$ for every $\gamma \in \sop(f)$.
\end{proposition}

\begin{proof}
Define the space
    $$E= \clspan \left(\{e_i: i \in \sop(f)\} \cup \{f\} \right)$$
where $e_i$, $i \in \Gamma$, are the canonical elements of $\linf{\Gamma}$. We have then  that $\dens(E)$ equals the cardinality of $\sop(f)$ so, since $f \in \linf{\kappa, \Gamma}$, we conclude that it will be strictly smaller than $\kappa$.

Since $X$ is a $\kappa$ ideal we get a bounded operator $P: E \ra X$ of norm equal to $1$ and acting as the identity operator on $E \cap X$. In particular, if $i \in \sop(f)$ we get $e_i \in E \cap c_0(\Gamma) \subset E \cap X$, so $P(e_i) = e_i$. Define $g = P(f)$ and let us show that $g$ satisfies our requirements.

Let $i \in \sop(f)$ and observe that
    \begin{align*}
        \norm{3\norm{f} e_i \pm f} & = \sup_{\gamma \in \Gamma} \abs{3\norm{f}e_i(\gamma) \pm f(\gamma)} = \max\{\abs{3 \norm{f} \pm f(i)}, \sup_{\gamma \in \Gamma \setminus\{i\}} \abs{f(\gamma)}\} \\
         & \leq \max\{\abs{3\norm{f} \pm f(i)}, \norm{f}\} = \abs{3\norm{f} \pm f(i)} = 3\norm{f} \pm f(i).
    \end{align*}
Notice also that
    \begin{align*}
        \norm{3\norm{f} e_i \pm f} = 3\norm{f} \pm f(i).
    \end{align*}
On the other hand, since $\norm{g} = \norm{P(f)} \leq \norm{P} \norm{f} = \norm{f}$, the same argument allows to conclude that
    \begin{align*}
        \norm{3\norm{f} e_i \pm g} = 3\norm{f} \pm g(i).
    \end{align*}
Moreover, we have that
    \begin{align*}
        3\norm{f} \pm g(i) & =\norm{3\norm{f} e_i \pm g} = \norm{3 \norm{f}P(e_i) \pm P(f)} \\
        & = \norm{P(3\norm{f}e_i \pm f)} \leq \norm{3\norm{f} e_i \pm f}  = 3\norm{f} \pm f(i)
    \end{align*}
from where the equality $f(i) = g(i)$ holds, as desired.
\end{proof}

\begin{corollary}\label{contrasimsyost}
Let $\Gamma$ be any subset with $|\Gamma| \geq \kappa$ and $X$ be a $\kappa^+$ ideal in $\ell_\infty(\Gamma)$ which contains $c_0(\Gamma)$. Then, $\dens(X) \geq 2^{\kappa}$.
\end{corollary}

\begin{proof}
Consider any proper subset $\Lambda \subset \Gamma$ with $|\Lambda| = \kappa$ and define the set  
$$A:= \left\{f\in \ell_\infty(\Gamma): \forall i\in \Lambda \,f(i)\in \{-1,1\} \land \forall i \in \Gamma \setminus \Lambda\, f(i) = 0\right\}.$$
Then $A \subset \ell_\infty(\kappa^+,\Gamma)$ and by Proposition~\ref{prop:kappaidealc0} we get that, given $f\in A$, there exists some $x_f\in X$ such that $x_f(i)=f(i)$ for every $i\in \supp(f)$. Observe that clearly the set $\{x_f: f\in A\}$ is such that $\Vert x_f-x_g\Vert \geq \Vert f-g\Vert = 2$ holds for every $f,g\in A, f\neq g$, so in particular $\dens(X)\geq\vert A\vert=2^{\kappa}$ as wanted. 
\end{proof}

\begin{remark}\label{rema:ejelinfikappaaiiideal}
An inspection in the proof of Proposition~\ref{prop:kappaidealc0} allows to obtain the following version for almost isometric ideals: Let $X$ be a $\kappa$ almost isometric ideal in $\ell_\infty(\Gamma)$ which contains $c_0(\Gamma)$. Then, for every $f \in \ell_\infty(\kappa, \Gamma)$ and every $\varepsilon>0$, there exists a function $g \in X$ such that $\vert g(\gamma)-f(\gamma)\vert<\varepsilon$ for every $\gamma \in \sop(f)$. Furthermore, is $X$ is a $\kappa^+$ almost isometric ideal, then $\dens(X) \geq 2^{\kappa}$.
\end{remark}

This last result proves that neither Theorem \ref{theo:simyost} with $\eps = 0$ nor Theorem~\ref{theo:abrahamsenai} can be proven in ZFC, in most cases, for the transfinite setting:

\begin{example}\label{exam:fallokappaideal}
Let $X = \ell_\infty(\kappa^+)$ and $Y = c_0(\kappa^+)$ for some successor cardinal $\kappa^+$. If there exists a $\kappa^+$ (almost isometric) ideal $Z$ such that $Y \subset Z \subset X$, then Corollary \ref{contrasimsyost} (or Remark \ref{rema:ejelinfikappaaiiideal}) implies that $\dens(Z) \geq 2^\kappa$ which is consistent to be greater than $\kappa^+ = \dens(Y)$.
\end{example}

\section{Almost isometric ideals in tensor product and ultrapower spaces}\label{section:aiidealestensorultrapower}

In this section we will study some results about almost isometric ideals in injective and projective tensor products. The reason for considering these results in a separate section is that we will need several perturbation arguments which are typical of finite dimensional Banach spaces and do not fit with the transfinite spirit of the previous sections. 

The following, which is the $\aleph_0$ version of Theorem \ref{theo:injetensortransai}, is the main result of this section.

\begin{theorem}\label{theorem:aiidealinyect}
Let $X,Y$ be Banach spaces and let $Z\subseteq X$ and $W\subseteq Y$ be almost isometric ideals. Then $Z\injtensor W$ is an almost isometric ideal in $X\injtensor Y$.
\end{theorem}

For the proof we will need a pair of auxiliary lemmata whose aim is to show that the condition of being an almost isometric ideal actually relies on having such property at the level of dense subspaces. 

\begin{lemma}\label{lema1teninj}
Let $X$ be a Banach space and let $Y \subset Z \subset X$ be subspaces such that $Z$ is dense in $X$ and assume that $Y$ is an almost isometric ideal in $Z$. Then $Y$ is an almost isometric ideal in $X$.
\end{lemma}

\begin{proof}
Let $E \subset X$ be a finite dimensional subspace and $\eps > 0$.

Since $E \cap Y$ is finite dimensional we can consider a finite basis of it, say $y_1, \cdots, y_n$, and we can enlarge it to a basis of $E$ $y_1, \cdots, y_n, e_{n+1}, \cdots, e_p$. Since $E$ is finite dimensional we get that $E$ is isomorphic to $\ell_1^p$, which gives us a constant $C>0$ such that, given $e \in E$, we get
\begin{align}\label{lema1tens1}
 \sum_{i=1}^p \vert \lambda_i\vert\leq C\left\Vert \sum_{i=1}^n \lambda_i y_i+\sum_{i=n+1}^p \lambda_i e_i\right\Vert\end{align}
holds for any choice of scalars $\lambda_1,\ldots,\lambda_p\in\mathbb R$.

On the other hand, by the density of $Z$ in $X$ we can find linearly independent vectors $z_{n+1}, \cdots, z_p \in Z$ such that, for every $i= n+1, \cdots, p$, we have
\begin{align}\label{lema1tens2}
    \norm{z_i - e_i} < \dfrac{\delta}{C}
\end{align}
where $0 < \delta < \min\{\sqrt{1+\eps} -1, 1- \sqrt{1-\eps}\}.$

Define
$$E' = \lspan\{y_1, \cdots, y_n, z_{n+1}, \cdots, z_p\} \subset Z.$$
Since $Y$ is an almost isometric ideal in $Z$ there exists a $\delta$-isometry $T': E' \ra Y$ which acts as the identity on $E' \cap Y$. Taking this into account, we can define $T: E \ra Y$ as the unique linear operator such that
\begin{align*}
    T(y_i) = & y_i, \, i= 1, \cdots, n, \\
    T(e_i) = & T'(z_i), \, i = n+1, \cdots, p.
\end{align*}
Let us prove that $T$ is the $\eps$-isometry that we were looking for. On the one hand, given $e \in E \cap Y$, we get that $e$ is a linear combination of $y_1, \cdots, y_n$, so the very definition of the operator $T$ yields $T(e) = e$.
  
On the other hand, in order to prove that $T$ is an $\varepsilon$-isometry, select any $e \in E$. We can express $e$ as
$$e = \sum_{i=1}^n \lambda_i y_i + \sum_{i=n+1}^p \lambda_i e_i,$$
and since $T'(y_i) = y_i$ holds for every $1\leq i\leq n$ (because $y_i \in E' \cap Y$), we get
    $$T(e) = \sum_{i=1}^n \lambda_i y_i + \sum_{i=n+1}^p \lambda_i T'(z_i) = T' \left( \sum_{i=1}^n \lambda_i y_i + \sum_{i=n+1}^p \lambda_i z_i\right).$$
A combination of equations (\ref{lema1tens1}) and  (\ref{lema1tens2}) yields
    \begin{align*}
        \norm{T(e)} & \leq (1+\delta) \norm{\sum_{i=1}^n \lambda_i y_i + \sum_{i=n+1}^p \lambda_i z_i} \leq (1+\delta) \left[ \norm{e} + \norm{\sum_{i=n+1}^p \lambda_i (z_i-e_i)} \right] \\
         & \mathop{\leq}\limits^{\mbox{\tiny{(\ref{lema1tens2})}}} (1+\delta) \left[ \norm{e} + \frac{\delta}{C} \sum_{i=1}^n\vert \lambda_i\vert \right] \mathop{\leq}\limits^{\mbox{\tiny{(\ref{lema1tens1})}}} (1+\delta)^2 \norm{e} < (1+\eps) \norm{e}.
    \end{align*}
With similar estimates we get that $\norm{T(e)}\geq (1-\eps) \norm{e}$.
This proves that $T$ is an $\varepsilon$-isometry and concludes the proof.
\end{proof}

Now we need a second auxiliary lemma which deals with another perturbation argument.

\begin{lemma}\label{lema2teninj}
Let $X$ be a Banach spaces and let $Y \subset Z \subset X$ be subspaces such that $Y$ is dense in $Z$ and such that $Y$ is an almost isometric ideal in $X$. Then $Z$ is an almost isometric ideal in $X$.
\end{lemma}

\begin{proof}
Let $E \subset X$ be a finite dimensional subspace and $\eps > 0$.

As in the proof of Lemma~\ref{lema1teninj}, by the condition that $E \cap Z$ is finite dimensional we can find a finite basis for it $z_1, \cdots, z_n$ which we can extend to a basis of the whole $E$, say $z_1, \cdots, z_n, e_{n+1}, \cdots, e_p$. Once again the fact that $E$ is finite dimensional implies that $E$ is isomorphic to $\ell_1^p$, which gives us a constant $C>0$ such that, given $e \in E$, we get
\begin{align}\label{lema2tens1}
 \sum_{i=1}^p \vert \lambda_i\vert\leq C\left\Vert \sum_{i=1}^n \lambda_i y_i+\sum_{i=n+1}^p \lambda_i e_i\right\Vert\end{align}
holds for any choice of scalars $\lambda_1,\ldots,\lambda_p\in\mathbb R$.

On the other hand, since $Y$ is dense in $Z$ we can take linearly independent vectors $y_{1}, \cdots, y_n \in Y$ such that, given $i= 1, \cdots, n$, we get
\begin{align}\label{lema2tens2}
    \norm{y_i - z_i} < \dfrac{\delta}{C}
\end{align}
for $0 < \delta$ such that $\delta^2+3\delta < \eps$ and $-\delta^2 + 3 \delta < \eps$.

Define
$$E' = \lspan\{z_1, \cdots, z_n, e_{n+1}, \cdots, e_p, y_1, \cdots, y_n\}.$$
Since $Y$ is an almost isometric ideal in $X$ there exists a $\delta$-isometry $T': E' \ra Y$ that acts as the identity operator on $E' \cap Y$, so in particular it fixes the points $y_1, \cdots, y_n$. This allows us to define $T: E \ra Z$ as the unique linear operator satisfying the following conditions.
\begin{align*}
    T(z_i) = & z_i, \, i= 1, \cdots, n, \\
    T(e_i) = & T'(e_i), \, i = n+1, \cdots, p.
\end{align*}
Let us prove that $T$ is the $\eps$-isometry that we were looking for. On the one hand, given $e \in E \cap Z$ it is clear that $e$ is linear combination of the elements $z_1, \cdots, z_n$, and from the definition of $T$ we infer $T(e) = e$.

On the other hand, let us prove that $T$ is an $\varepsilon$-isometry. In order to do so select any $e \in E$ and write it as
    $$e = \sum_{i=1}^n \lambda_i z_i + \sum_{i=n+1}^p \lambda_i e_i$$
so that    
    $$T(e)=\sum_{i=1}^n \lambda_i z_i+\sum_{i=n+1}^p\lambda_i T'(e_i).$$
    By (\ref{lema2tens1}) and (\ref{lema2tens2}) we obtain that
    \begin{align*}
        \norm{T(e)} & \leq \norm{\sum_{i=1}^n \lambda_i (z_i - y_i)} + \norm{\sum_{i=1}^n \lambda_i y_i + \sum_{i=n+1}^p \lambda_i T'(e_i)} \\
         & = \norm{\sum_{i=1}^n \lambda_i (z_i - y_i)} + \norm{T' \left(\sum_{i=1}^n \lambda_i y_i + \sum_{i=n+1}^p \lambda_i e_i\right)} \\
         & \mathop{\leq}\limits^{\mbox{\tiny{(\ref{lema2tens2})}}}\frac{\delta}{C} \sum_{i=1}^n\vert \lambda_i\vert + (1+\delta) \norm{\sum_{i=1}^n \lambda_i y_i + \sum_{i=n+1}^p \lambda_i e_i}  \\
         & \mathop{\leq}\limits^{\mbox{\tiny{(\ref{lema2tens1})}}} \delta \norm{e} + (1+\delta) \left[  \norm{e} + \norm{\sum_{i=1}^n \lambda_i (z_i - y_i)} \right]  \\
         & \leq (\delta^2 + 3 \delta + 1) \norm{e} < (1+ \eps) \norm{e}.
    \end{align*}
In a similar way, we get $\norm{T(e)}> (1 - \eps) \norm{e}$,
so that $T$ is the $\varepsilon$-isometry we were looking for.
\end{proof}

Now we are ready to provide the proof of Theorem~\ref{theorem:aiidealinyect}.

\begin{proof}[Proof of Theorem~\ref{theorem:aiidealinyect}]
Observe that it is enough to prove that $Z \otimes_\eps W$ is an almost isometric ideal in $X \otimes_\eps Y$, because in such case Lemma~\ref{lema1teninj} implies that $Z \otimes_\eps W$ is an almost isometric ideal in $X \injtensor Y$ and then an application of Lemma~\ref{lema2teninj} implies that $Z \injtensor W$ is an almost isometric ideal in $X \injtensor Y$.

Consequently let us prove that $Z \otimes_\eps W$ is an almost isometric ideal in $X \otimes_\eps Y$, for which we take a finite dimensional subspace $E \subset X \otimes_\eps Y$ and $\eps >0$. Consider a basis $\bar{z_1}, \cdots, \bar{z_n}$ of $E \cap (Z \otimes_\eps W)$ and let us consider an extension to a basis of the whole space $E$, say $\bar{z_1}, \cdots, \bar{z_n}, \bar{e}_{n+1}, \cdots, \bar{e_p}$. By definition, every $\bar{z_i}$ can be written as finite sum of basic tensors in the following way
$$\bar{z_i} = \sum_{k=1}^{n_i} z_{ki} \otimes w_{ki},$$
where $z_{ki} \in Z$ and $w_{ki} \in W$. In a similar way we may write $\bar{e_i}$ in the form
$$\bar{e_i} = \sum_{k=1}^{n_i} x_{ki} \otimes y_{ki}$$
for certain $x_{ki} \in X$ and $y_{ki} \in Y$.

Now we consider the following finite dimensional spaces
\begin{align*}
    U = & \lspan\{z_{ki}; x_{ki}\} \subset X, \\
    V = & \lspan\{w_{ki}; y_{ki}\} \subset Y.
\end{align*}
Since $Z$ and $W$ are almost isometric ideals in $X$ and $Y$ respectively we can find two $\delta$-isometries $T: U \ra Z$ and $S: V \ra W$, for $\delta >0$ small enough to get $(1+\delta)^2 < 1+ \eps$ and $(1-\delta)^2 > 1-\eps$, acting both as the identity operator on $U \cap Z$ and on $V \cap W$ respectively. Since $T, S$ are $\delta$-isometries, for every $z \in U \otimes_\eps V$ we get that
$$(1-\eps) \norm{z} < (1-\delta)^2 \norm{z} \leq \norm{(T \otimes S)(z)} \leq (1+\delta)^2 \norm{z} < (1+\eps) \norm{z},$$
so $T \otimes S$ is an $\eps$-isometry too. It remains to prove that $T\otimes S$ acts as the identity operator on $E \cap (Z \otimes_\eps W)$ to conclude that $(T \otimes S)|_E$ is the operator that we were looking for.

In order to show this, select any $e \in E \cap (Z \otimes_\eps W)$, and observe that $e$ admits the following expression as finite sum of elementary tensors
    $$e = \sum_{i=1}^n \lambda_i \bar{z_i} = \sum_{i=1}^n \sum_{k=1}^{n_i} \lambda_i (z_{ki} \otimes w_{ki})$$
  for certain $\lambda_i\in\mathbb R$. Since $z_{ik} \in Z \cap U$ and $w_{ki} \in V \cap W$ holds for every $i,k$ we infer that
    $$(T \otimes S)(e) = \sum_{i=1}^n \sum_{k=1}^{n_i} \lambda_i (T(z_{ki}) \otimes S(w_{ki})) = \sum_{i=1}^n \sum_{k=1}^{n_i} \lambda_i (z_{ki} \otimes w_{ki}) = e$$
as desired.
\end{proof}

Now several comments are pertinent.

\begin{remark}\label{remark:consequencestensor}
In the first place, as we are concerned, Theorem~\ref{theo:injetensortransai} was unknown even for the $\aleph_0$ case. The authors acknowledge Vegard Lima addressing us this question.

On the other hand, in the proof of Theorem~\ref{theorem:aiidealinyect} the perturbation arguments behind the proofs of Lemmata~\ref{lema1teninj} and \ref{lema2teninj} are crucial, and it does not seem possible to get a suitable version of Theorem~\ref{theorem:aiidealinyect} for the case of isometric ideals. However, we do not know whether such version of Theorem~\ref{theorem:aiidealinyect} holds true.
\end{remark}

Now we move to the question whether Theorem~\ref{theo:injetensortransai} may hold for the projective tensor product, i.e. if it is possible that when $E\subseteq X$ and $F\subseteq Y$ are $\kappa$ almost isometric ideals, then $E\projtensor F$ is an almost isometric ideal in $X\projtensor Y$. Our aim is to present in Example~\ref{exam:contraproyect} that the answer to the above question is negative even if we consider $F=Y$ and for the classical notion of almost isometric ideal. We will show this fact by making use of an indirect argument which have to do with octahedral norms, for which we will present the following proposition. Due to the complex flavour of Example~\ref{exam:contraproyect} let us consider in the following proposition the scalar field to be either $\mathbb R$ or $\mathbb C$.

\begin{proposition}\label{prop:octahtensorproideal}
Let $X$ be an octahedral Banach space and $Y$ be any non-zero Banach space over the scalar field $\mathbb K$, being $\mathbb K$ either $\mathbb R$ or $\mathbb C$. If $(X\times \{0\})\projtensor Y$ is an almost isometric ideal in $(X\oplus_1\mathbb K)\projtensor Y$, then $X\projtensor Y$ is octahedral.
\end{proposition}

\begin{proof}
Let us prove that $(X\times\{0\})\projtensor Y$ (which is isometrically isomorphic to $X\projtensor Y$) is octahedral. Select $z_1,\ldots, z_n\in S_{(X \times\{0\})\projtensor Y}$ and $\varepsilon>0$, and let us prove that there exists $z\in S_{(X\times\{0\})\projtensor Y}$ such that
$$\Vert z_i+z\Vert \ge 2-\varepsilon\ \forall 1\leq i\leq n.$$
This is enough to get that $(X\times \{0\})\projtensor Y$ is octahedral (c.f. e.g. \cite[Proposition 3.3 (iii)]{lan}). 

In order to do so select $y\in S_Y$ and $y^*\in S_{Y^*}$ such that $y^*(y)=1$ and fix $1\leq i\leq n$. Select an element $T_i\in ((X\times\{0\})\projtensor Y)^*=L(X\times\{0\},Y^*)$ such that $T_i(z_i)=1$ and $\norm{T_i} = 1$. Define $G_i:X\oplus_1\mathbb K\longrightarrow Y^*$ by the equation
$$G_i((x,\lambda)):=T_i(x,0)+\lambda y^*.$$
We claim that $\Vert G_i\Vert_{L(X\oplus_1\mathbb K,Y^*)}\leq 1$. Indeed, given any $x\in X$ and $\lambda\in\mathbb K$ we get
\[
\begin{split}
\Vert G_i((x,\lambda))\Vert=\Vert T_i(x,0)+\lambda y^*\Vert& \leq \Vert T_i(x,0)\Vert+\vert\lambda\vert\leq \Vert T_i\Vert \Vert (x,0)\Vert+\vert\lambda\vert\\
& =\Vert x\Vert+\vert\lambda\vert=\Vert (x,\lambda)\Vert_{X\oplus_1\mathbb K}.    
\end{split}
\]
Consequently $\Vert G_i\Vert=1$. Observe also that $G_i(z_i)=T_i(z_i)$. In fact, observe that $G_i=T_i$ agree on $(X\times\{0\})\projtensor Y$ (observe that $(X\times\{0\})\projtensor Y$ is a subspace of $(X\oplus_1\mathbb K)\projtensor Y$ since $X\times\{0\}$ is an almost isometric ideal in $X\oplus_1 \mathbb K$ by (the complex version of) Corollary~\ref{coro:aiiealoctaclasisuma}). Consequently
$$\Vert z_i+(0,1)\otimes y\Vert\geq G_i(z_i)+G_i((0,1))(y)=1+y^*(y)=2.$$
The arbitrariness of $1\leq i\leq n$ implies that
$$\Vert z_i+(0,1)\otimes y\Vert_{(X\oplus_1\mathbb K)\pten Y}=2$$
holds for every $1\leq i\leq n$. In order to finish the proof, since $(X\times 0)\projtensor Y$ is an almost isometric ideal in $(X\oplus_1\mathbb K)\projtensor Y$ we can find an $(\varepsilon/2)$-isometry $T:\mathrm{span}\{z_1,\ldots, z_n,(0,1)\otimes y\}\longrightarrow (X\times\{0\})\projtensor Y$ which fixes the elements of the intersection (in particular $T(z_i)=z_i$ holds for $1\leq i\leq n$). Now the element $z:=T((0,1)\otimes y)/\norm{T((0,1) \otimes y)}$ satisfies our purposes.
\end{proof}

The following example shows that if $Y$ is an almost isometric ideal in $X$ and $Z$ is any Banach space then $Y\projtensor Z$ is not necessarily an almost isometric ideal in $X\projtensor Z$.

\begin{example}\label{exam:contraproyect}
In \cite[Theorem 3.2]{rueda25} it is proved that there exists a 2-dimensional complex space $Y\subseteq \ell_\infty^{10}(\mathbb C)$ such that, for $X=L_\infty^\mathbb C([0,1])$, the space $X\projtensor Y$ fails to be octahedral, where $L_\infty^\mathbb C([0,1])$ stands for the space of all complex-valued $L_\infty$ functions. Since $X$ is octahedral, we conclude that $X\times \{0\}$ is an almost isometric ideal in $X\oplus_1\mathbb K$ (Corollary~\ref{coro:aiiealoctaclasisuma}) but $(X\times\{0\})\projtensor Y$ is not an almost isometric ideal in $(X\oplus_1\mathbb K)\projtensor Y$ (Proposition~\ref{prop:octahtensorproideal}).   
\end{example}

We end this section with a result about preservance of almost isometric ideals by taking ultrapower spaces. We refer the reader to Subsection~\ref{subsect:ultrapowers} for necessary notation about ultrapower spaces.

\begin{proposition}\label{prop:ultrapoweraiideals}
Let $X$ be a Banach space and let $Y$ be a subspace of $X$. The following assertions are equivalent:
\begin{enumerate}
    \item $Y$ is an almost isometric ideal in $X$.
    \item $Y_\mathcal{U}$ is an isometric ideal in $X_\mathcal{U}$ for any countably incomplete ultrafilter $\mathcal{U}$ on any infinite set $I$.
\end{enumerate}
\end{proposition}

\begin{proof}
(1) $\Rightarrow$ (2)

Let $E \subset X_\mathcal{U}$ be a finite dimensional subspace and consider a basis for $E$
$$\{z_1, \cdots,z_p, v_{p+1}, \cdots, v_q\}$$
with the property that $\{z_1, \cdots, z_p\}$ is a basis of $E \cap Y_\mathcal{U}$ (if $E \cap Y_\mathcal{U}$ is trivial simply ignore the $z_i$ in the proof). Moreover we select representatives for the elements $z_j = [y_i^j]_\mathcal{U}$ for every $1 \leq j \leq p$, where $y_i^j \in Y$ and $v_j = [x_i^j]_\mathcal{U}$ for $p+1 \leq j \leq q$.

Now define, for every $i\in I$, a finite dimensional subspace of $X$ by
$$E_i = \lspan\{y_i^1, \cdots,y_i^p, x_i^{p+1}, \cdots, x_i^q\}.$$
By our assumption there exists a bounded operator $T_i: E_i \ra Y$ such that
\begin{align}\label{ultra1}
    T_i(y_i^j) = y_i^j, \, j=1, \cdots, p
\end{align}
and such that
\begin{align}\label{ultra2}
    (1-h(i)) \norm{e} \leq \norm{T_i(e)} \leq (1+h(i)) \norm{e}, \, \forall e \in E_i,
\end{align}
where $h:I \ra \R$ is a strictly positive function such that $\lim_{\mathcal{U}} h(i) = 0$ (observe that one such $h$ does exists since $\mathcal{U}$ is countably incomplete, see e.g. \cite[Proposition 1.1.8]{grelier} for a proof).

Now we can define $T: E \ra Y_\mathcal{U}$ as the unique linear operator satisfying the following conditions
\begin{align*}
    & T(z_j) = [T_i(y_i^j)]_\mathcal{U} = [y_i^j]_\mathcal{U} = z_j, \, j= 1, \cdots, p, \\
    & T(v_j) = [T_i(x_i^j)]_\mathcal{U}, \, j= p+1, \cdots, q ,
\end{align*}
and let us prove that $T$ is the mapping that we were looking for. On the one hand, if $e \in E \cap Y_\mathcal{U}$, we can write 
$$e = \sum_{j=1}^p \alpha_j z_j$$
for certain $\alpha_j\in\mathbb R$, so
$$T(e) = \sum_{j=1}^p \alpha_j T(z_j) = \sum_{j=1}^p \alpha_j z_j = e.$$
The above proves that $T$ acts as the identity operator on $E \cap Y_\mathcal{U}$.

On the other hand, let us prove that $T$ is an isometry. To do so, select any $e \in E$, and consider an expression of $e$ as linear combination of the elements of the basis $\{z_1,\ldots, z_p, v_{p+1},\ldots, v_q\}$, that is
    $$e = \sum_{j=1}^p \alpha_j z_j + \sum_{j=p+1}^q \alpha_j v_j = \left[\sum_{j=1}^p \alpha_j y_i^j + \sum_{j=p+1}^q \alpha_j x_i^j  \right]_\mathcal{U}$$
Using the linearity of $T$ we infer
\begin{align*}
        T(e) = & \sum_{j=1}^p \alpha_j T(z_j) + \sum_{j=p+1}^q \alpha_j T(v_j) = \sum_{j=1}^p \alpha_j [T_i(y_i^j)]_\mathcal{U} + \sum_{j=p+1}^q \alpha_j [T_i(x_i^j)]_\mathcal{U} \\
         = & \left[T_i\left(\sum_{j=1}^p \alpha_j y_i^j + \sum_{j=p+1}^q \alpha_j x_i^j \right) \right]_\mathcal{U}.
\end{align*}
Fix $\eps > 0$ and define the following sets
    \begin{align*}
        A = & \{i\in I: h(i) < \eps\}, \\
        B = & \left\{i\in I: \abs{\norm{\sum_{j=1}^p \alpha_j y_i^j + \sum_{j=p+1}^q \alpha_j x_i^j} - \norm{e}} < \eps \norm{e}\right\}, \\
        C = & \left\{i\in I: \abs{\norm{T_i\left(\sum_{j=1}^p \alpha_j y_i^j + \sum_{j=p+1}^q \alpha_j x_i^j\right)} - \norm{T(e)}} < \eps \norm{e}\right\}. \\
    \end{align*}
Observe that the three above elements belong to $\mathcal{U}$ since $\lim_\mathcal{U} f(i) = 0$, $\norm{e} = \lim_\mathcal{U} \norm{e_i}$ and $\norm{T(e)} = \lim_\mathcal{U} \norm{T_i(e_i)}$.
If we select $i \in A \cap B \cap C$ we get, taking into account equation (\ref{ultra2}), that
    \begin{align*}
        \norm{T(e)} & > \norm{T_i\left(\sum_{j=1}^p \alpha_j y_i^j + \sum_{j=p+1}^q \alpha_j x_i^j\right)} - \eps\norm{e} \\
         & \geq  (1-h(i)) \norm{\sum_{j=1}^p \alpha_j y_i^j + \sum_{j=p+1}^q \alpha_j x_i^j} - \eps \norm{e} \\
         & > (1-\eps)^2 \norm{e} - \eps\norm{e} = \left( (1-\eps)^2 - \eps \right)\norm{e}.
    \end{align*}
The condition $\norm{T(e)}<\left( (1+\eps)^2 + \eps \right)\norm{e}$ is proved in a similar way.
Summarising we have proved that, for every $e \in E$, the following condition holds
    $$\left( (1-\eps)^2 - \eps \right)\norm{e} \leq \norm{T(e)} \leq \left( (1+\eps)^2 + \eps \right)\norm{e}.$$
The arbitrariness of $\varepsilon$ show that $T$ is an isometry, as desired.

(2) $\Rightarrow$ (1)

Let $E \subset X$ be a finite dimensional subspace and let $\eps >0$.

Define $F= j(E)$,  where $j: X \ra X_\mathcal{U}$ is the canonical injection and, since $F$ is finite dimensional, we get by the assumptions that there exists an isometry $\phi: F \ra Y_\mathcal{U}$ acting as the identity operator on $F \cap Y_\mathcal{U}$.

On the other hand, $j(Y)$ is an almost isometric ideal in $Y_\mathcal{U}$ (see the two last paragraphs of Subsection~\ref{subsect:ultrapowers}). Since $\phi(F)$ is a finite dimensional subspace of $Y_\mathcal{U}$ there exists a $\eps$-isometry $\Phi: \phi(F) \ra j(Y)$ acting as the identity operator on $\phi(F) \cap j(Y)$.

Finally define $T: E \ra Y$ by $\psi = j^{-1} \circ \Phi \circ \phi \circ j \circ \iota$ (where $\iota: E \ra X$ is the inclusion operator). Let us prove that $T$ is the $\eps$-isometry that we were looking for. On the one hand, given $e \in E \cap Y$, we have that
    $$j(\iota (e)) = [e]_\mathcal{U} \in F \cap Y_\mathcal{U},$$
    so
    $$\phi(j(\iota(e))) = j(e) \in \phi(F) \cap j(Y).$$
The above allows us to get,
    $$\Phi(\phi(j(\iota(e)))) = \phi(j(\iota(e))) = j(e),$$
from where we obtain that
    $$T(e) = j^{-1}(j(e)) = e.$$
    On the other hand, since $j$ is an isometry we get, given $e \in E$ that
    \begin{align*}
        \norm{T(e)} = \norm{\Phi(\phi(j(e)))} \leq (1+\eps) \norm{\phi(j(e))} = (1+\eps) \norm{e} \\
        \norm{T(e)} = \norm{\Phi(\phi(j(e)))} \geq (1-\eps) \norm{\phi(j(e))} = (1-\eps) \norm{e} \\
    \end{align*}
This concludes (1) and finishes the proof.
\end{proof}

\section{Concluding remarks and open questions}\label{sect:openquestion}

In this section we collect some open questions derived from our work.

To begin with, as we stated in the remark after Definition \ref{defi:transideal} we did not find any example of spaces $Y\subseteq X$ such that $Y$ is $\kappa$ almost ideal in $X$ but which is not a $\kappa$ ideal in $X$ out of the case that $\kappa=\aleph_0$. If such counterexample existed, then the study of this new property would have perfect sense by its own. Because of this reason, we wonder the following:

\begin{question}
If $Y\subseteq X$ is a $\kappa$ almost ideal for some uncountable cardinal $\kappa$, does it imply that $Y$ is a $\kappa$ ideal in $X$?
\end{question}

As we pointed of in the paragraph preceding Proposition~\ref{theo:idealfree}, it is known that if $\mathcal F(Y)$ is an ideal in $\mathcal F(X)$ then $Y$ is an ideal in $X$, which follow by making use of weak-star compactness arguments which are impossible to apply to our transfinite setting. Because of this, we do not know whether the above result extends to a transfinite version in the following sense:

\begin{question}\label{question:lipfree}
Let $Y\subseteq X$ be Banach spaces and $\kappa$ an uncountable cardinal. If $\mathcal F(Y)$ is a $\kappa$ ideal in $\mathcal F(X)$, is it true that $Y$ is a $\kappa$ ideal in $X$? 
\end{question}

The examples constructed in Example \ref{exam:fallokappaideal} are strongly based on the break of the Continuum Hypothesis. Because of this, we wonder:

\begin{question}\label{question:contraejesimyostgeneral}
Can we construct counterexamples with the properties of Example \ref{exam:fallokappaideal} without the addition of any axiom to ZFC? Can we extend the counterexamples to any cardinal $\kappa$?
\end{question}

From point (2) in Remark~\ref{remark:consequencestensor} the following question arises:

\begin{question}
Let $X,Z$ be Banach spaces and let $Y\subseteq X$ be an isometric ideal in $X$. Is it true that $Y\injtensor Z$ is an isometric ideal in $X\injtensor Y$?
\end{question}

\section*{Acknowledgements}  

The authors are grateful with Vegard Lima for answering inquires about almost isometric ideals and for addresing the authors the question whether the injective tensor product preserves almost isometric ideal, which motivated the authors the search of Theorem~\ref{theorem:aiidealinyect}. The authors are also deeply grateful with Gin\'es L\'opez-P\'erez for fruitful and regular discussions on the topic of the paper.

Part of the research of the topic was developed during the participation of the last author to the congress “Structures in Banach Spaces” which took place at the Erwin Schr\"odinger International Institute for Mathematics and Physics (ESI) of the University of Vienna, institution which supported part of the expenses. The last author is grateful to the ESI for the financial support received.

This research has been supported  by MCIU/AEI/FEDER/UE\\  Grant PID2021-122126NB-C31, by MICINN (Spain) Grant \\ CEX2020-001105-M (MCIU, AEI) and by Junta de Andaluc\'{\i}a Grant FQM-0185. The research of E. Mart\'inez Va\~n\'o has also been supported by Grant PRE2022-101438 funded by MCIN/AEI/10.13039/501100011033 and ESF+. The research of A. Rueda Zoca was also supported by Fundaci\'on S\'eneca: ACyT Regi\'on de Murcia grant 21955/PI/22.

\end{document}